\documentclass [12pt]{amsart}
\usepackage{amssymb,amsmath,amsthm}
\newtheorem{theorem}{Theorem}
\newtheorem{lemma}{Lemma}

\newcommand{\ad}{\,\mathrm{ad}\,}

\newcommand{\GL}{\,\mathrm{GL}\,}
\newcommand{\SL}{\,\mathrm{SL}\,}

\newcommand{\diag}{\,\mathrm{diag}\,}
\begin{document}

\begin{center}

{\Large {\bf Automorphisms of Chevalley groups

of type  $F_4$ over local rings with $1/2$\footnote{The
work is supported by the Russian President grant MK-2530.2008.1 and
by the grant of Russian Fond of Basic Research 08-01-00693.} }}

\bigskip
\bigskip

{\large \bf E.~I.~Bunina}

\end{center}
\bigskip

\begin{center}

{\bf Abstract.}

\end{center}

In the given paper we prove that every automorphism of a Chevalley group of type $F_4$ over a commutative local ring with~$1/2$ is standard, i.\,e., it is a composition of ring and inner automorphisms.

\bigskip

\section*{Introduction}\leavevmode

An associative commutative ring $R$ with a unit is called  \emph{local},
if it contains exactly one maximal ideal (that coincides with the radical of~$R$). Equivalently, the set of all non-invertible
elements of~$R$ is an ideal.

We describe automorphisms of Chevalley groups of type $F_4$ over local rings with~$1/2$. Note that for the root system $F_4$ there exists only one weight lattice, that is simultaneously universal and adjoint, therefore for every ring~$R$ there exists a unique Chevalley group of type $F_4$, that is
 $G(R)=G_{\ad}(F_4,R)$. Over local rings universal Chevalley groups coincide with their elementary subgroups, consequently the Chevalley group $G(R)$ is also an elementary Chevalley group.

   Theorem~1 for the root systems $A_l, D_l, $ and $E_l$ was obtained by the author in~\cite{ravnyekorni}, in~\cite{normalizers} all automorphisms of Chevalley groups of given types over local rings with~$1/2$ were described.  Theorem~1 for the root systems $B_2$ and $G_2$ is proved in~\cite{korni2}.

 Similar results for Chevalley groups over fields were proved
 by R.\,Steinberg~\cite{Stb1} for the finite case and by J.\,Humphreys~\cite{H} for the infinite case. Many papers were devoted
to description of automorphisms of Chevalley groups over different
commutative rings, we can mention here the papers of
Borel--Tits~\cite{v22}, Carter--Chen~Yu~\cite{v24},
Chen~Yu~\cite{v25}--\cite{v29}, A.\,Klyachko~\cite{Klyachko}.
 E.\,Abe~\cite{Abe_OSN} proved that all automorphisms of Chevalley groups under Noetherian  rings with~$1/2$ are standard.

The case
$A_l$ was completely studied by the papers of
W.C.~Waterhouse~\cite{v46}, V.M.~Petechuk~\cite{v12},  Fuan Li and
Zunxian Li~\cite{v37}, and also for rings without~$1/2$. The paper
of I.Z.\,Golubchik and A.V.~Mikhalev~\cite{v8} covers the
case~$C_l$, that is not considered in the present paper. Automorphisms and isomorphisms of general linear groups over arbitrary associative rings were described by E.I.~Zelmanov in~\cite{v11} and by I.Z.~Golubchik, A.V.~Mikhalev in~\cite{GolMikh1}.

 We generalize some methods of V.M.~Petechuk~\cite{Petechuk1} to prove Theorem~1.

The author is thankful to N.A.\,Vavilov,  A.A.\,Klyachko,
A.V.\,Mikhalev for valuable advices, remarks and discussions.

\section{Definitions and main theorems.}\leavevmode

 We fix the root system~$\Phi$ of the type $F_4$ (detailed texts about root
systems and their properties can be found in the books
\cite{Hamfris}, \cite{Burbaki}). Let $e_1,e_2,e_3,e_4$ be an orthonorm basis of the space $\mathbb R^4$. Then we  numerate the roots of $F_4$
as follows:
$$
\alpha_1=e_2-e_3, \alpha_2=e_3-e_4, \alpha_3=e_4,
\alpha_4=\frac{1}{2}(e_1-e_2-e_3-e_4)
$$
are simple roots;
\begin{align*}
\alpha_5&=\alpha_1+\alpha_2=e_2-e_4,\\
\alpha_6&=\alpha_2+\alpha_3=e_3,\\
\alpha_7&=\alpha_3+\alpha_4=\frac{1}{2}(e_1-e_2-e_3+e_4),\\
\alpha_8&=\alpha_1+\alpha_2+\alpha_3=e_2,\\
\alpha_9&=\alpha_2+\alpha_3+\alpha_4=\frac{1}{2}(e_1-e_2+e_3-e_4),\\
\alpha_{10}&=\alpha_2+2\alpha_3=e_3+e_4,\\
\alpha_{11}&=\alpha_1+\alpha_2+\alpha_3+\alpha_4=\frac{1}{2}(e_1+e_2-e_3-e_4),\\
\alpha_{12}&=\alpha_1+\alpha_2+2\alpha_3=e_2+e_4,\\
\alpha_{13}&=\alpha_2+2\alpha_3+\alpha_4=\frac{1}{2}(e_1-e_2+e_3+e_4),\\
\alpha_{14}&=\alpha_1+2\alpha_2+2\alpha_3=e_2+e_3,\\
 \alpha_{15}&=
\alpha_1+\alpha_2+2\alpha_3+\alpha_4=\frac{1}{2}(e_1+e_2-e_3+e_4),\\
\alpha_{16}&=\alpha_2+2\alpha_3+2\alpha_4= e_1-e_2,\\
\alpha_{17}&=\alpha_1+2\alpha_2+2\alpha_3+\alpha_4=\frac{1}{2}(e_1+e_2+e_3-e_4),\\
\alpha_{18}&= \alpha_1+\alpha_2+2\alpha_3+2\alpha_4=e_1-e_3,\\
\alpha_{19}&=\alpha_1+2\alpha_2+3\alpha_3+\alpha_4=\frac{1}{2}(e_1+e_2+e_3+e_4),\\
\alpha_{20}&=\alpha_1+2\alpha_2+2\alpha_3+2\alpha_4=e_1-e_4,\\
\alpha_{21}&=\alpha_1+2\alpha_2+3\alpha_3+2\alpha_4=e_1,\\
\alpha_{22}&=
\alpha_1+2\alpha_2+4\alpha_3+2\alpha_4=e_1+e_4,\\
\alpha_{23}&= \alpha_1+3\alpha_2+4\alpha_3+2\alpha_4=e_1+e_3,\\
\alpha_{24}&=2\alpha_1+3\alpha_2+4\alpha_3+2\alpha_4=e_1+e_2
\end{align*}
are other positive roots.

Suppose now that we have a
semisimple complex Lie algebra~$\mathcal L$ of type $F_4$ with
Cartan subalgebra~$\mathcal H$ (detailed information about
semisimple Lie algebras can be found in the book~\cite{Hamfris}).

 Then in the algebra $\mathcal L$ we can choose a \emph{Chevalley basis}
  $\{ h_i\mid i=1,\dots,4; x_\alpha\mid \alpha\in \Phi\}$ so that for
every two elements of this basis their commutator is an integral
linear combination of the elements of the same basis.

Namely,

1) $[h_i,h_j]=0;$

2) $[h_i,x_\alpha]=\langle \alpha_i,\alpha\rangle x_\alpha$;

3) if $\alpha=n_1\alpha_1+\dots+n_4\alpha_4$, then
$[x_{\alpha},x_{-\alpha}]=n_1h_1+\dots+n_4h_4$;

4) if $\alpha+\beta\notin \Phi$, then $[x_{\alpha},x_{\beta}]=0$;

5) if $\alpha+\beta\in \Phi$, and $\alpha,\beta$ are roots of the same length, then $[x_\alpha,x_\beta]=c x_{\alpha+\beta}$;

6) if $\alpha+\beta\in \Phi$, $\alpha$ is a long root, $\beta$
is a short root, then $[x_{\alpha},x_\beta]=a x_{\alpha+\beta}+b
x_{\alpha+2\beta}$.

Take now an arbitrary local ring with $1/2$ and construct an elementary adjoint Chevalley group of type $F_4$ over this ring (see, for example~\cite{Steinberg}). For our convenience we briefly put here the construction.

In the Chevalley basis of $\mathcal L$ all operators
$(x_\alpha)^k/k!$ for $k\in \mathbb N$ are written
as integral (nilpotent) matrices. An integral matrix also can be considered as
a matrix over an arbitrary commutative ring with~$1$.Let $R$ be
such a ring. Consider matrices $n\times n$ over~$R$, matrices $(x_\alpha)^k/k!$ for
 $\alpha\in \Phi$, $k\in \mathbb N$ are included in $M_n(R)$.

Now consider automorphisms of the free module $R^n$ of the form
$$
\exp (tx_\alpha)=x_\alpha(t)=1+tx_\alpha+t^2 (x_\alpha)^2/2+\dots+
t^k (x_\alpha)^k/k!+\dots
$$
Since all matrices $x_\alpha$ are nilpotent, we have that this
series is finite. Automorphisms $x_\alpha(t)$ are called
\emph{elementary root elements}. The subgroup in $Aut(R^n)$,
generated by all $x_\alpha(t)$, $\alpha\in \Phi$, $t\in R$, is
called an \emph{elementary adjoint Chevalley group} (notation:
$E_{\ad}(\Phi,R)=E_{\ad}(R)$).

In an elementary Chevalley group there are the following important elements:

--- $w_\alpha(t)=x_\alpha(t) x_{-\alpha}(-t^{-1})x_\alpha(t)$, $\alpha\in \Phi$,
$t\in R^*$;

--- $h_\alpha (t) = w_\alpha(t) w_\alpha(1)^{-1}$.

The action of $x_\alpha(t)$ on the Chevalley basis is described in
\cite{v23}, \cite{VavPlotk1}, we write it below (see Section~3).

Over local rings for the root system $F_4$ all Chevalley groups coincide with elementary adjoint Chevalley groups $E_{\ad}(R)$,
 therefore we do not introduce Chevalley
groups themselves in this paper. In this paper we denote our Chevalley groups by $G(R)$, since they depend only of a ring~$R$.

We will work with two types of standard automorphisms of a Chevalley group
 $G(R)$ and with one unusual, ``temporary'' type of automorphisms.

{\bf Ring automorphisms.} Let $\rho: R\to R$ be an automorphism of
the ring~$R$. The mapping $x\mapsto \rho (x)$ from $G(R)$
onto itself is an automorphism of the group $G(R)$, that is
denoted by the same letter~$\rho$ and is called a \emph{ring
automorphism} of the group~$G(R)$. Note that for all
$\alpha\in \Phi$ and $t\in R$ an element $x_\alpha(t)$ is mapped to
$x_\alpha(\rho(t))$.

{\bf Inner automorphisms.} Let $g\in G(R)$ be an element of a Chevalley group under consideration.
Conjugation of the group $G(R)$ with the element~$g$ is an automorphism of  $G(R)$, that is denoted by~$i_g$
and is called an  \emph{inner automorphism} of~$G(R)$.

These two types of automorphisms are called \emph{standard}. There are  central and graph automorphisms, which are also standard, but in our case (root system $F_4$) they can not appear. Therefore we say that an automorphism of the group
 $G(R)$ is standard, if it is a composition of ring and inner automorphisms.

Besides that, we need also to introduce  temporarily one more type of automorphisms:

{\bf Automorphisms--conjugations.} Let $V$ be a representation space of the Chevalley group $G(R)$, $C\in \GL(V)$ be  a matrix from the normalizer of $G(R)$:
$$
C G(R) C^{-1}= G (R).
$$
 Then the mapping $x\mapsto CxC^{-1}$ from $G(R)$ onto itself is an automorphism of the Chevalley group, which  is denoted by $i_С$ and is called an \emph{automorphism--conjugation} of~$G(R)$,
\emph{induced by the element}~$C$ of the group~$\GL(V)$.

\smallskip

In Section~5 we will prove that in our case all automorphisms--conjugations are inner, but the first step is the proof of the following theorem:

\begin{theorem}\label{first} Let $G(R)$ be a Chevalley group of type $F_4$, where $R$
is a commutative local ring with~$1/2$. Then every automorphism of $G(R)$ is a composition of a ring automorphism and an automorphism--conjugation.
\end{theorem}

Sections 2--4 are  devoted to the proof of Theorem~1.

\section{Changing the initial automorphism to a special isomorphism, images of~$w_{\alpha_i}$}\leavevmode

Since in the papers \cite{ravnyekorni} and~\cite{korni2} the root system in there second sections was arbitrary, we can suppose all results of these sections to be proved also for our root system~$F_4$.

Namely, by the fixed automorphism $\varphi$ we can construct a mapping $\varphi'= i_{g^{-1}}  \varphi$, which is an isomorphism of the group  $G(R)\subset \GL_n(R)$ onto some
subgroup of $\GL_n(R)$ with the property that its image under factorization $R$ by $J$ (the radical of~$R$) coinsides with a ring automorphism $\overline \rho$.

Besides, from sections~2 of the same papers we know that the image of any involution (a matrix of order~$2$) under suah an isomorphism is conjugate to this involution in the group $\GL_n(R)$.

These are the main facts that we need to know.

The order of roots we have fixed in the previous section.

The basis of the space~$V$ ($52$-dimensional) we numerate as
$v_i=x_{\alpha_i}$, $v_{-i}=x_{-\alpha_i}$, $V_1=h_{1}$,\dots ,
$V_4=h_4$.

Consider the matrices $h_{\alpha_1}(-1),\dots, h_{\alpha_4}(-1)$ in our
basis. They have the form
\begin{align*}
h_{\alpha_1}(-1)&=\diag [1,1,-1,-1,1,1,1,1, -1,-1,-1,-1,1,1,
 -1,-1,-1,-1,-1,-1,
  -1,-1,\\
  &\quad -1,-1,-1,-1,1,1,
 -1,-1,-1,-1,1,1,
-1,-1,1,1,1,1,1,1,1,1,-1,-1,-1, -1, 1,1,1,1],\\
h_{\alpha_2}(-1)&=\diag [-1,-1,1,1,-1,-1,1,1, -1,-1,-1,-1,-1,-1,
 1,1,-1,-1,1,1,
  1,1,-1,-1,\\
  & \quad 1,1,-1,-1,
 -1,-1,1,1,-1,-1,
-1,-1,1,1,-1,-1,1,1,-1,-1,-1,-1,1,1,1,1,1,1],\\
h_{\alpha_3}(-1)&=\diag [1,1,1,1,1,1,-1,-1, 1,1,1,1,-1,-1,
 1,1,-1,-1,1,1,\\
 & \quad -1,-1,1,1,-1,-1,1,1,
 -1,-1,1,1,-1,-1,
1,1,-1,-1,1,1,1,1,1,1,1,1,1,1,1,1,1,1],\\
h_{\alpha_4}(-1)&=\diag [1,1,1,1,-1,-1,1,1, 1,1,-1,-1,-1,-1,
 -1,-1,-1,-1,1,1,\\
 & \quad -1,-1,1,1,1,1,1,1,
 1,1,1,1,1,1,
1,1,-1,-1,1,1,-1,-1,1,1,1,1,1,1,1,1,1,1].
\end{align*}
As we see, for all~$i$ we have
$h_{\alpha_i}(-1)^2=1$.

We know that every matrix $h_i=\varphi'(h_{\alpha_i}(-1))$ in some basis is diagonal with $\pm 1$ on its diagonal, and
the number of  $1$ and
$-1$ coincides with their number for the matrix $h_{\alpha_i}(-1)$. Since all matrices $h_i$ commute, then there exists a basis, where all
$h_i$ has the same form as $h_{\alpha_i}(-1)$ in the initial basis from weight vectors. Suppose that we came to this basis with the help of the matrix~$g_1$. Clear that
 $g_1\in \GL_n(R,J)=\{ X\in \GL_n(R)\mid X-E\in M_n(J)\}$. Consider the mapping
 $\varphi_1=i_{g_1}^{-1} \varphi'$. It is also an isomorphism of the group
 $G(R)$ onto some subgroup of $\GL_n(R)$ such that its image under factorization $R$ by~$J$ is~$\overline \rho$, and
$\varphi_1(h_{\alpha_i}(-1))=h_{\alpha_i}(-1)$ for all
$i=1,\dots,4$.

Instead of $\varphi'$ we now consider the isomorphism~$\varphi_1$.

Every element $w_i=w_{\alpha_i}(1)$ moves by conjugation  $h_i$ to each other, therefore its image has a block-monomial form. In particular, this image can be rewritten as a block-diagonal matrix, where the first block is $48\times 48$, and the second is $4\times 4$.

Consider the first basis vector after the last basis change. Denote it by~$e$. The Weil group  $W$
acts transitively on the set of roots of the same length, therefore for every
root~$\alpha_i$ of the same length as the first one,
there exists such $w^{(\alpha_i)}\in W$, that $w^{(\alpha_i)}
\alpha_1=\alpha_i$. Similarly, all roots of the second length are also conjugate under the action
of~$W$. Let $\alpha_k$ be the first root of the length that is not equal to the length of~$\alpha_1$, and let $f$ be
the $k$-th basis vector after the last basis change. If $\alpha_j$ is a root conjugate to $\alpha_k$, then let us denote by $w_{(\alpha_j)}$ an element of~$W$ such that
$w_{(\alpha_j)} \alpha_k=\alpha_j$. Consider now the basis
$e_1,\dots, e_{48}, e_{49},\dots, e_{52}$, where $e_1=e$, $e_k=f$,
for $1< i\leqslant 48$ either $e_i=\varphi_1(w^{(\alpha_i)})e$, or
$e_i=\varphi_1(w_{(\alpha_i)})f$ (it depends of the length of $\alpha_k$); for $48< i\leqslant 52$ we do not move $e_i$. Clear that the matrix of this basis change is equivalent to the unit modulo radical. Therefore the obtained set of vectors also is a basis.

Clear that a matrix for $\varphi_1(w_i)$ ($i=1,\dots,4$) in the basis part $\{ e_1,\dots,e_{2n}\}$
 coincides with the matrix for $w_i$ in the initial basis of weight vectors.
Since $h_i(-1)$ are squares of $w_i$, then there images are not changed in the new basis.

Besides, we know that every matrix $\varphi_1(w_i)$
is block-diagonal up to decomposition of basis in the first  $48$
and last $4$ elements. Therefore the last part of basis consisting of $4$ elements, can be changed independently.

Initially (in the basis of weight vectors)  $w_i$ in this basis part are

$$
w_1:
\begin{pmatrix}
-1& 1& 0&  0\\
0& 1& 0& 0\\
0& 0& 1 &0\\
0& 0& 0&  1
\end{pmatrix},\quad
w_2:
\begin{pmatrix}
1& 0& 0&  0\\
1& -1& 1& 0\\
0& 0& 1&0\\
0& 0& 0&  1
\end{pmatrix},\quad
w_3:
\begin{pmatrix}
1& 0& 0& 0\\
0& 1& 0& 0\\
0& 2& -1&1\\
0& 0& 0&  1
\end{pmatrix},\quad
w_4:
\begin{pmatrix}
1& 0& 0& 0\\
0& 1& 0& 0\\
0& 0& 1&0\\
0& 0& 1& -1
\end{pmatrix}.
$$

We have the following conditions for these elements (on the given basis part):

1) for all $i$ $w_i^2=E$;

2) $w_i$ and $w_j$ commute for $|i-j|> 1$;

3) $w_1 w_2$ and $w_3w_4$ have order~$3$, $w_2w_3$ has order~$2$.

Therefore the images
$\varphi_1(w_i)$ satisfy the same conditions. Besides, we know, that these images are  equivalent to the initial~$w_i$
modulo radical~$J$.

Let us make the basis change with the matrix, which is a product of (commuting with each other) matrices
$$
\begin{pmatrix}
1& 1/2& 0&  0\\
0& 1& 0&  0\\
0& 0& 1& 0\\
0& 0& 0& 1
\end{pmatrix}\text{ и }
\begin{pmatrix}
1& 0& 0&  0\\
0& 1& 0&  0\\
0& 1& 1& 1\\
0& 0& 0& 2
\end{pmatrix}.
$$

In this basis $w_1=\diag[-1,1,1,1]$, $w_3 =\diag[1,1,-1,1]$,
$$
w_2=\begin{pmatrix}
1/2& 1/4& -1/2& -1/2\\
1& 1/2& 1&1\\
-1& 1/2& 0&-1\\
0& 0& 0& 1
\end{pmatrix},
\quad w_4=\begin{pmatrix}
1& 0& 0& 0\\
0& 1& 0&0\\
0& -1/2& 1/2& 3/2\\
0& 1/2& 1/2& -1/2
\end{pmatrix}.
$$

Consider now the images of $\varphi_1(w_i)$ in the changed basis. All these images are involutions, and every of them has exactly one $-1$ in its diagonal form, also $\varphi_1(w_1)$ and $\varphi_1(w_3)$
commute. Hence we can choose such a basis (equivalent to the previous one modulo~$J$), where
$\varphi_1(w_1)$ and $\varphi_1(w_3)$ have a diagonal form with one
$-1$ on the corresponding places.

Consider now where  $w_4$ can move under this basis change.

Since $\varphi_1(w_4)$ commutes with $\varphi_1(w_1)$, has order two and is equivalent to $w_4$
modulo radical, we have
$$
\varphi_1(w_4)=\begin{pmatrix}
1& 0& 0& 0\\
0& a& b& c\\
0& d& e& f\\
0& g& h& i
\end{pmatrix}.
$$
Use the facts that  $\varphi_1(w_4)^2=E$, $\varphi_1(w_3w_4)$ has order~$3$. Then we obtain
$$
\begin{cases}
ad+de+fg=0,\\
ad-de+fg=-d
\end{cases},
$$
therefore $2de=d$,  and since $d\equiv 1/2 \mod J$, we have $e=1/2$. Moreover,
$$
\begin{cases}
ag+dh+gi=0,\\
ag-dh+gi=g
\end{cases}
$$
consequently $2g(a+i)=g$, i.\,e., $a+i=1/2$. Make now a basis change with the matrix
$$
\begin{pmatrix}
1& 0& 0& 0\\
0& 1& 0& \frac{a-1}{g}\\
0& 0& 1& 0\\
0& 0& 0& 1
\end{pmatrix}.
$$
This change does not move the elements $\varphi_1(w_1)$ and $\varphi_1(w_3)$,
and  $\varphi_1(w_4)$ now has the form
$$
\begin{pmatrix}
1& 0& 0& 0\\
0& 1& b& c\\
0& d& 1/2& f\\
0& g& h& -1/2
\end{pmatrix}.
$$
Using the above conditions, we obtain the equation
$bg+eh+hi=0$, consequently $bg=0$, i.\,e., $b=0$. In this case from $a^2+bd+cg=1$ it follows $c=0$.
All other conditions gives the system
$$
\begin{cases}
fg=-3/2d,\\
dh=-g/2,\\
fh=-1/4.
\end{cases}
$$
Clear that with a diagonal basis change (which does not move
$\varphi_1(w_1)$ and $\varphi_1(w_3)$) we can come to a basis, where
$\varphi_1(w_4)$ has the same form as $w_4$ after the first our basis change. Making now the inverse basis change, we obtain that
$\varphi_1(w_1)$, $\varphi_1(w_3)$ and $\varphi(w_4)$ have the same form as $w_1,w_3,w_4$, respectively. Look at $\varphi_1(w_2)$.

Since $\varphi_1(w_2)$ commutes with $\varphi_1(w_4)$,
we have
$$
\varphi_1(w_2)=
\begin{pmatrix}
a&b&c& 0\\
d& e& f& 0\\
g& i& h& 0\\
g/2& i/2& k& h-2k
\end{pmatrix}.
$$
Since $(h-2k)^2=1$, we have $h-2k=1$. Now similarly to the consideration of $\varphi_1(w_4)$, we take the conditions for $\varphi_1(w_2)$. After suitable diagonal change we get $\varphi_1(w_i)=w_i$ in the new last basis.

Therefore we can now come from the isomorphism~$\varphi_1$ under consideration to an isomorphism $\varphi_2$ with all properties of $\varphi_1$ and such that
$\varphi_2(w_i)=w_i$ for all $i=1,\dots,4$.

We suppose now that an isomorphism~$\varphi_2$ with all these properties is given.

\section{Images of $x_{\alpha_i}(1)$ and diagonal matrices.}

Let us write the matrices $w_i$, $i=1,\dots,4$:
{\small
\begin{multline*}
w_1= -e_{\alpha_1,-\alpha_1}-e_{-\alpha_1,\alpha_1}+e_{\alpha_2,\alpha_5}+e_{-\alpha_2,-\alpha_5}-e_{\alpha_5,\alpha_2}-e_{-\alpha_5,\alpha_2}+\\
+e_{\alpha_3,\alpha_3}+e_{-\alpha_3,-\alpha_3}+e_{\alpha_4,\alpha_4}+e_{-\alpha_4,-\alpha_4}+e_{\alpha_6,\alpha_8}+e_{-\alpha_6,-\alpha_8}-
e_{\alpha_8,\alpha_6}-e_{-\alpha_8,-\alpha_6}+\\
+e_{\alpha_7,\alpha_7}+e_{-\alpha_7,-\alpha_7}+e_{\alpha_9,\alpha_{11}}+e_{-\alpha_9,-\alpha_{11}}-
e_{\alpha_{11},\alpha_9}-e_{-\alpha_{11},-\alpha_9}+\\
+e_{\alpha_{10},\alpha_{12}}+e_{-\alpha_{10},-\alpha_{12}}-
e_{\alpha_{12},\alpha_{10}}-e_{-\alpha_{12},-\alpha_{10}}+e_{\alpha_{13},\alpha_{15}}+e_{-\alpha_{13},-\alpha_{15}}-
e_{\alpha_{15},\alpha_{13}}-e_{-\alpha_{15},-\alpha_{13}}+\\
+e_{\alpha_{14},\alpha_{14}}+e_{-\alpha_{14},-\alpha_{14}}+e_{\alpha_{16},\alpha_{18}}+e_{-\alpha_{16},-\alpha_{18}}-
e_{\alpha_{18},\alpha_{16}}-e_{-\alpha_{18},-\alpha_{16}}+\\
+e_{\alpha_{17},\alpha_{17}}+e_{-\alpha_{17},-\alpha_{17}}+e_{\alpha_{19},\alpha_{19}}+e_{-\alpha_{19},-\alpha_{19}}+
e_{\alpha_{20},\alpha_{20}}+e_{-\alpha_{20},-\alpha_{20}}+\\
+e_{\alpha_{21},\alpha_{21}}+e_{-\alpha_{21},-\alpha_{21}}
+e_{\alpha_{22},\alpha_{22}}+e_{-\alpha_{22},-\alpha_{22}}+
+e_{\alpha_{23},\alpha_{24}}+e_{-\alpha_{23},-\alpha_{24}}-\\
-e_{\alpha_{24},\alpha_{23}}-e_{-\alpha_{24},-\alpha_{23}}-e_{h_1,h_1}+e_{h_1,h_2}+e_{h_2,h_2}+e_{h_3,h_3}+e_{h_4,h_4};
\end{multline*}

\begin{multline*}
w_2= -e_{\alpha_2,-\alpha_2}-e_{-\alpha_2,\alpha_2}+e_{\alpha_1,\alpha_5}+e_{-\alpha_1,-\alpha_5}-e_{\alpha_5,\alpha_1}-e_{-\alpha_5,\alpha_1}+\\
-e_{\alpha_3,\alpha_6}-e_{-\alpha_3,-\alpha_6}+
e_{\alpha_6,\alpha_3}+e_{-\alpha_6,-\alpha_3}+e_{\alpha_4,\alpha_4}+e_{-\alpha_4,-\alpha_4}+\\
+e_{\alpha_7,\alpha_9}+e_{-\alpha_7,-\alpha_9}-
e_{\alpha_9,\alpha_7}-e_{-\alpha_9,-\alpha_7}
+e_{\alpha_8,\alpha_8}+e_{-\alpha_8,-\alpha_8}+e_{\alpha_{10},\alpha_{10}}+e_{-\alpha_{10},-\alpha_{10}}+
e_{\alpha_{11},\alpha_{11}}+e_{-\alpha_{11},-\alpha_{11}}+\\
+e_{\alpha_{12},\alpha_{14}}+e_{-\alpha_{12},-\alpha_{14}}-
e_{\alpha_{14},\alpha_{12}}-e_{-\alpha_{14},-\alpha_{12}}+e_{\alpha_{13},\alpha_{13}}+e_{-\alpha_{13},-\alpha_{13}}+\\
+e_{\alpha_{15},\alpha_{17}}+e_{-\alpha_{15},-\alpha_{17}}-
e_{\alpha_{17},\alpha_{15}}-e_{-\alpha_{17},-\alpha_{15}}+e_{\alpha_{16},\alpha_{16}}+e_{-\alpha_{16},-\alpha_{16}}+\\
+e_{\alpha_{18},\alpha_{20}}+e_{-\alpha_{18},-\alpha_{20}}-
e_{\alpha_{20},\alpha_{18}}-e_{-\alpha_{20},-\alpha_{18}}+e_{\alpha_{19},\alpha_{19}}+e_{-\alpha_{19},-\alpha_{19}}+\\
+e_{\alpha_{21},\alpha_{21}}+e_{-\alpha_{21},-\alpha_{21}}+e_{\alpha_{22},\alpha_{23}}+e_{-\alpha_{22},-\alpha_{23}}-
e_{\alpha_{23},\alpha_{22}}-e_{-\alpha_{23},-\alpha_{22}}+\\
+e_{\alpha_{24},\alpha_{24}}+e_{-\alpha_{24},-\alpha_{24}}+e_{h_1,h_1}+e_{h_2,h_1}-e_{h_2,h_2}+e_{h_2,h_3}+e_{h_3,h_3}+e_{h_4,h_4};
\end{multline*}

\begin{multline*}
w_3= e_{\alpha_1,-\alpha_1}+e_{-\alpha_1,\alpha_1}+e_{\alpha_2,\alpha_{10}}+e_{-\alpha_2,-\alpha_{10}}+e_{\alpha_{10},\alpha_2}+e_{-\alpha_{10},\alpha_2}-\\
-e_{\alpha_3,-\alpha_3}-e_{-\alpha_3,\alpha_3}+e_{\alpha_4,\alpha_7}+e_{-\alpha_4,-\alpha_7}-
e_{\alpha_7,\alpha_4}-e_{-\alpha_7,-\alpha_4}+\\
+e_{\alpha_5,\alpha_{12}}+e_{-\alpha_5,-\alpha_{12}}+
e_{\alpha_{12},\alpha_5}+e_{-\alpha_{12},-\alpha_5}
-e_{\alpha_6,\alpha_6}-e_{-\alpha_6,-\alpha_6}-e_{\alpha_{8},\alpha_{8}}-e_{-\alpha_{8},-\alpha_{8}}+
\\
+e_{\alpha_{9},\alpha_{13}}+e_{-\alpha_{9},-\alpha_{13}}-
e_{\alpha_{13},\alpha_{9}}-e_{-\alpha_{13},-\alpha_{9}}+e_{\alpha_{11},\alpha_{15}}+e_{-\alpha_{11},-\alpha_{15}}-
e_{\alpha_{15},\alpha_{11}}-e_{-\alpha_{15},-\alpha_{11}}+\\
+e_{\alpha_{14},\alpha_{14}}+e_{-\alpha_{14},-\alpha_{14}}+e_{\alpha_{16},\alpha_{16}}+e_{-\alpha_{16},-\alpha_{16}}+\\
+e_{\alpha_{17},\alpha_{19}}+e_{-\alpha_{17},-\alpha_{19}}-
e_{\alpha_{19},\alpha_{17}}-e_{-\alpha_{19},-\alpha_{17}}+e_{\alpha_{18},\alpha_{18}}+e_{-\alpha_{18},-\alpha_{18}}+\\
+e_{\alpha_{20},\alpha_{22}}+e_{-\alpha_{20},-\alpha_{22}}+
e_{\alpha_{22},\alpha_{20}}+e_{-\alpha_{22},-\alpha_{20}}-e_{\alpha_{21},\alpha_{21}}-e_{-\alpha_{21},-\alpha_{21}}+e_{\alpha_{23},\alpha_{23}}
+e_{-\alpha_{23},-\alpha_{23}}+\\
+e_{\alpha_{24},\alpha_{24}}+e_{-\alpha_{24},-\alpha_{24}}+e_{h_1,h_1}+e_{h_2,h_2}+2e_{h_3,h_2}-e_{h_3,h_3}+e_{h_3,h_4}+e_{h_4,h_4};
\end{multline*}

\begin{multline*}
w_4= e_{\alpha_1,-\alpha_1}+e_{-\alpha_1,\alpha_1}+e_{\alpha_2,-\alpha_2}+e_{-\alpha_2,\alpha_2}-e_{\alpha_3,\alpha_{7}}
-e_{-\alpha_3,-\alpha_{7}}
+e_{\alpha_{7},\alpha_3}+e_{-\alpha_{7},\alpha_3}-\\
-e_{\alpha_4,-\alpha_4}-e_{-\alpha_4,\alpha_4}+e_{\alpha_5,\alpha_5}+e_{-\alpha_5,-\alpha_5}
+e_{\alpha_6,\alpha_{9}}+e_{-\alpha_6,-\alpha_{9}}-
e_{\alpha_{9},\alpha_6}-e_{-\alpha_{9},-\alpha_6}+\\
+e_{\alpha_8,\alpha_{11}}+e_{-\alpha_8,-\alpha_{11}}-e_{\alpha_{11},\alpha_{8}}-e_{-\alpha_{11},-\alpha_{8}}
+e_{\alpha_{10},\alpha_{16}}+e_{-\alpha_{10},-\alpha_{16}}+
e_{\alpha_{16},\alpha_{10}}+e_{-\alpha_{16},-\alpha_{10}}+\\
+e_{\alpha_{12},\alpha_{18}}+e_{-\alpha_{12},-\alpha_{18}}+
e_{\alpha_{18},\alpha_{12}}+e_{-\alpha_{18},-\alpha_{12}}-
-e_{\alpha_{13},\alpha_{13}}-e_{-\alpha_{13},-\alpha_{13}}-e_{\alpha_{15},\alpha_{15}}-e_{-\alpha_{15},-\alpha_{15}}-\\
-e_{\alpha_{17},\alpha_{17}}-e_{-\alpha_{17},-\alpha_{17}}+e_{\alpha_{14},\alpha_{20}}+e_{-\alpha_{14},-\alpha_{20}}+
e_{\alpha_{20},\alpha_{14}}+e_{-\alpha_{20},-\alpha_{14}}+\\
+e_{\alpha_{19},\alpha_{21}}+e_{-\alpha_{19},-\alpha_{21}}-
e_{\alpha_{21},\alpha_{19}}-e_{-\alpha_{21},-\alpha_{19}}+e_{\alpha_{22},\alpha_{22}}+e_{-\alpha_{22},-\alpha_{22}}+e_{\alpha_{23},\alpha_{23}}
+e_{-\alpha_{23},-\alpha_{23}}+\\
+e_{\alpha_{24},\alpha_{24}}+e_{-\alpha_{24},-\alpha_{24}}+e_{h_1,h_1}+e_{h_2,h_2}+e_{h_3,h_3}+e_{h_4,h_3}-e_{h_4,h_4}.
\end{multline*}

}

Besides that, $x_{\alpha_1}(t)=E+tX_1+t^2X_1^2/2$, where
\begin{multline*}
X_1=2e_{\alpha_1,h_1}-e_{\alpha_1,h_2}-e_{h_1,-\alpha_1}+e_{\alpha_5,\alpha_2}-e_{-\alpha_2,-\alpha_5}
+e_{\alpha_8,\alpha_6}-e_{-\alpha_6,-\alpha_8}+\\
+e_{\alpha_{11},\alpha_9}-e_{-\alpha_9,-\alpha_{11}}+e_{\alpha_{12},\alpha_{10}}-e_{-\alpha_{10},-\alpha_{12}}+e_{\alpha_{15},\alpha_{13}}-
e_{-\alpha_{13},-\alpha_{15}}+\\
+e_{\alpha_{18},\alpha_{16}}-e_{-\alpha_{16},-\alpha_{18}}+e_{\alpha_{24},\alpha_{23}}-e_{-\alpha_{23},-\alpha_{24}};
\end{multline*}
$x_{\alpha_3}(t)=E+tX_3+t^2X_3^2/2$, where
\begin{multline*}
X_3=-2e_{\alpha_3,h_2}+2e_{\alpha_3,h_3}-e_{\alpha_3,h_4}-e_{h_3,-\alpha_3}+e_{\alpha_7,\alpha_4}-e_{-\alpha_4,-\alpha_7}+\\
+e_{\alpha_{13},\alpha_9}-e_{-\alpha_9,-\alpha_{13}}+e_{\alpha_{15},\alpha_{11}}-e_{-\alpha_{11},-\alpha_{15}}+e_{\alpha_{19},\alpha_{17}}
-e_{-\alpha_{17},-\alpha_{19}}-\\
-2e_{\alpha_6,\alpha_2}+e_{-\alpha_2,-\alpha_6}-e_{\alpha_{10},\alpha_6}+2e_{-\alpha_6,-\alpha_{10}}-2e_{\alpha_8,\alpha_5}+e_{-\alpha_5,-\alpha_8}-\\
-e_{\alpha_{12},\alpha_8}+2e_{-\alpha_8,\alpha_{12}}-2e_{\alpha_{21},\alpha_{20}}+e_{-\alpha_{20},-\alpha_{21}}-e_{\alpha_{22},\alpha_{21}}+
2e_{-\alpha_{21},-\alpha_{22}}.
\end{multline*}

We are interested in images of $x_{\alpha_i}(t)$. Let
$\varphi_2(x_{\alpha_1}(1))=x_1=(y_{i,j})$. Since $x_1$ commutes with all
$h_{\alpha_i}(-1)$, $i=1,3,4$, and also with $w_3$, $w_4$,  and $w_{14}$,  then by direct calculus we obtain:

1. The matrix  $x_1$ can be decomposed into following eight diagonal blocks:
\begin{align*}
B_1&=\{ v_{1}, v_{-1}, v_{{14}}, v_{-{14}},v_{{20}}, v_{-{20}}, v_{{22}}, v_{-{22}}, V_{1}, V_{2}, V_{3}, V_{4}\};\\
 B_2&=\{ v_{2}, v_{-2}, v_{5}, v_{-5}, v_{{10}}, v_{-{10}}, v_{{16}}, v_{-{16}},   v_{{18}}, v_{-{18}}, v_{{23}}, v_{-{23}}, v_{{24}}, v_{-{24}}\};\\
 B_3&=\{v_{3}, v_{-3}, v_{{21}}, v_{-{21}}\};\\
  B_4&=\{ v_{4}, v_{-4}, v_{{17}}, v_{-{17}}\};\\
B_5&=\{ v_{6}, v_{-6}, v_{8}, v_{-8}\};\\
B_6&=\{ v_{7}, v_{-7}, v_{{19}}, v_{-{19}}\};\\
B_7&=\{ v_{9}, v_{-9}, v_{{11}}, v_{-{11}}\};\\
B_8&=\{ v_{{13}}, v_{-{13}}, v_{{15}}, v_{-{15}}\}.
\end{align*}

2. On the block $B_1$ the matrix $x_1$ has the form
{\tiny
$$
\left(\begin{array}{cccccccccccc}
y_1& y_2& -y_3& y_3& -y_3& y_3& -y_3& y_3& -2y_4& y_4& 0& 0\\
y_5& y_6& -y_7& y_7& -y_7& y_7& -y_7& y_7& -2y_8& y_8& 0& 0\\
y_9& y_{10}& y_{11}& y_{12}& -y_{13}& y_{13}& -y_{13}& y_{13}& -2y_{14}+2y_{15}&
y_{14}& 0& -y_{15}\\
-y_9& -y_{10}& y_{12}& y_{11}& y_{13}& -y_{13}& y_{13}& -y_{13}& 2y_{14}-2y_{15}&
-y_{14}+2y_{15}& 0& y_{15}\\
y_9& y_{10}& -y_{13}& y_{13}& y_{11}& y_{12}& -y_{13}& y_{13}& 2(-y_{14}+y_{15})& y_{14}& -y_{15}& y_{15}\\
-y_9& -y_{10}& y_{13}& -y_{13}& y_{12}& y_{11}& y_{13}& -y_{13}& 2(y_{14}-y_{15})& -y_{14}+2y_{15}& -y_{15}& y_{15}\\
y_9& y_{10}& -y_{13}& y_{13}& -y_{13}& y_{13}& y_{11}& y_{12}&  2(-y_{14}+y_{15})& -y_{14}+2y_{15}& y_{15}& 0\\
-y_9& -y_{10}& y_{13}& -y_{13}& y_{13}& -y_{13}& y_{12}& y_{11}&  2(y_{14}-y_{15})& -y_{14}& y_{15}& 0\\
y_{16}& y_{17}& -y_{18}& y_{18}& -y_{18}& y_{18}& -y_{18}& y_{18}& y_{19}-2y_{20}& y_{20}& 0& 0\\
0& 0& 0& 0& 0& 0& 0& 0& 0& y_{20}& 0& 0\\
0& 0& 0& 0& 0& 0& 0& 0& 0& 0& y_{20}& 0\\
0& 0& 0& 0& 0& 0& 0& 0& 0& 0& 0& y_{20}
\end{array}\right).
$$
}

3. On the block $B_2$ it is
{\tiny
$$
\left(\begin{array}{cccccccccccccccc}
y_{21}&  y_{22}& -y_{23}& -y_{24}& -y_{25}& -y_{26}& y_{27}& y_{28}& -y_{25}& -y_{26}&
y_{27}& y_{28}& y_{28}& y_{27}& y_{26}& y_{25}\\
y_{29}& y_{30}& -y_{31}& -y_{32}& -y_{25}& -y_{33}& y_{34}& y_{35}& -y_{25}& -y_{33}&
y_{34}& y_{35}& y_{35}& y_{34}& y_{33}& y_{28}\\
y_{32}& y_{31}& y_{30}& y_{29}& -y_{35}& -y_{34}& -y_{33}& -y_{25}& -y_{35}& -y_{34}&
-y_{33}& -y_{25}& -y_{25}& -y_{33}& y_{34}& y_{35}\\
y_{24}& y_{23}& y_{22}& y_{21}& -y_{28}& -y_{27}& -y_{26}& -y_{25}& -y_{28}& -y_{27}&
-y_{26}& -y_{25}& -y_{25}& -y_{26}& y_{27}& y_{28}\\
-y_{25}& -y_{26}& y_{27}& y_{28}& y_{21}& y_{22}& -y_{23}& -y_{24}& -y_{25}& -y_{26}& y_{27}& y_{28}& y_{28}& y_{27}& y_{26}& y_{25}\\
-y_{25}& -y_{33}& y_{34}& y_{35}& y_{29}& y_{30}& -y_{31}& -y_{32}& -y_{25}& -y_{33}& y_{34}& y_{35}& y_{35}& y_{34}& y_{33}& y_{25}\\
-y_{35}& -y_{34}& -y_{33}& -y_{25}& y_{32}& y_{31}& y_{30}& y_{29}& -y_{35}& -y_{34}& -y_{33}& -y_{25}& -y_{25}& -y_{33}& y_{34}& y_{35}\\
-y_{28}& -y_{27}& -y_{26}& -y_{25}& y_{24}& y_{23}& y_{22}& y_{21}& -y_{28}& -y_{27}& -y_{26}& -y_{25}& -y_{25}& -y_{26}& y_{27}& y_{28}\\
-y_{25}& -y_{26}& y_{27}& y_{28}& -y_{25}& -y_{26}& y_{27}& y_{28}& y_{21}& y_{22}& -y_{23}& -y_{24}& y_{28}& y_{27}& y_{26}& y_{25}\\
-y_{25}& -y_{33}& y_{34}& y_{35}& -y_{25}& -y_{33}& y_{34}& y_{35}& y_{29}& y_{30}& -y_{31}& -y_{32}& y_{35}& y_{34}& y_{33}& y_{25}\\
-y_{35}& -y_{34}& -y_{33}& -y_{25}& -y_{35}& -y_{34}& -y_{33}& -y_{25}& y_{32}& y_{31}&
y_{30}& y_{29}& -y_{25}& -y_{33}& y_{34}& y_{35}\\
-y_{28}& -y_{27}& -y_{26}& -y_{25}& -y_{28}& -y_{27}& -y_{26}& -y_{25}& y_{24}& y_{23}&
y_{22}& y_{21}& -y_{25}& -y_{26}& y_{27}& y_{28}\\
-y_{28}& -y_{27}& -y_{26}& -y_{25}& -y_{28}& -y_{27}& -y_{26}& -y_{25}& -y_{28}& -y_{27}&
-y_{26}& -y_{25}& y_{21}& y_{22}& -y_{23}& -y_{24}\\
-y_{35}& -y_{34}& -y_{33}& -y_{25}& -y_{35}& -y_{34}& -y_{33}& -y_{25}& -y_{35}& -y_{34}&
-y_{33}& -y_{25}& y_{29}& y_{30}& -y_{31}& -y_{32}\\
y_{25}& y_{33}& -y_{34}& -y_{35}& y_{25}& y_{33}& -y_{34}& -y_{35}& y_{25}& y_{33}&
-y_{34}& -y_{35}& y_{32}& y_{31}& y_{30}& y_{29}\\
y_{25}& y_{26}& -y_{27}& -y_{28}& y_{25}& y_{26}& -y_{27}& -y_{28}& y_{25}& y_{26}&
-y_{27}& -y_{28}& y_{24}& y_{23}& y_{22}& y_{21}
\end{array}\right).
$$
}

4. On the blocks $B_3$, $B_4$, $B_6$ it has the form
$$
\begin{pmatrix}
y_{36}& y_{37}& y_{38}& y_{38}\\
y_{37}& y_{36}& y_{38}& y_{38}\\
-y_{38}& -y_{38}& y_{36}& y_{37}\\
-y_{38}& -y_{38}& y_{37}& y_{36}
\end{pmatrix}.
$$

5. Finally, on the blocks $B_5$, $B_7$, $B_8$ it is
$$
\begin{pmatrix}
y_{39}& y_{40}& y_{41}& y_{42}\\
y_{43}& y_{44}& y_{45}& y_{46}\\
-y_{46}& -y_{45}& y_{44}& y_{43}\\
-y_{42}& -y_{41}& y_{39}& y_{40}
\end{pmatrix}.
$$

Let now
$\varphi_2(x_{\alpha_4}(1))=x_4=(z_{i,j})$. Since $x_4$ commutes with all
$h_{\alpha_i}(-1)$, $i=1,2,4$, and $w_{1}$, $w_{2}$, and also for $w_{{13}}$ we have $w_{13} x_4 w_{13}^{-1}= x_4^{-1}=h_{\alpha_3}(-1) x_4 h_{\alpha_3}(-1)$,  then by direct calculation we obtain:

1. The matrix  $x_4$ can be decomposed into following eight diagonal blocks:
\begin{align*}
B_1'&=\{ v_{4}, v_{-4},  V_{1}, V_{2}, V_{3}, V_{4}\};\\
 B_2'&=\{ v_{1}, v_{-1}, v_{{14}}, v_{-{14}}, v_{{17}}, v_{-{17}}, v_{{20}}, v_{-{20}},   v_{{22}}, v_{-{22}}\};\\
B_3'&=\{ v_{2}, v_{-2}, v_{{10}}, v_{-{10}}, v_{{13}}, v_{-{13}}, v_{{16}}, v_{-{16}},   v_{{24}}, v_{-{24}}\};\\
B_4'&=\{ v_{5}, v_{-5}, v_{{12}}, v_{-{12}}, v_{{15}}, v_{-{15}}, v_{{18}}, v_{-{18}},   v_{{23}}, v_{-{23}}\};\\
 B_5'&=\{v_{6}, v_{-6}, v_{{9}}, v_{-{9}}\};\\
  B_6'&=\{ v_{3}, v_{-3}, v_{{7}}, v_{-{7}}\};\\
B_7'&=\{ v_{8}, v_{-8}, v_{{11}}, v_{-{11}}\};\\
B_8'&=\{ v_{19}, v_{-19}, v_{{21}}, v_{-{21}}\}.
\end{align*}

2. On the first block the matrix $x_4$ has the form
$$
\begin{pmatrix}
z_1& z_2& 0& 0& z_3& -2z_3\\
z_4& z_5& 0& 0& z_6& -2z_6\\
0& 0& z_7& 0& 0& 0\\
0& 0& 0& z_7& 0& 0\\
0& 0& 0& 0& z_7& 0\\
z_8& z_9& 0& 0& z_{10}& z_7-2z_{10}
\end{pmatrix}.
$$

3. On the second, third and fourth blocks it is
$$
\begin{pmatrix}
z_{11}& z_{12}& -z_{13}& -z_{14}& z_{15}& z_{15}& z_{14}& z_{13}& -z_{16}& z_{16}\\
z_{12}& z_{11}& z_{13}& z_{14}& -z_{15}& -z_{15}& -z_{14}& -z_{13}& z_{16}& -z_{16}\\
-z_{17}& z_{17}& z_{18}& z_{19}& z_{20}& z_{21}& z_{22}& z_{23}& z_{17}& -z_{17}\\
z_{24}& -z_{24}& z_{25}& z_{26}& z_{27}& z_{28}& z_{29}& z_{30}& -z_{24}& z_{24}\\
-z_{31}& z_{31}& z_{32}& z_{33}& z_{34}& z_{35}& z_{36}& z_{37}& z_{31}& -z_{31}\\
-z_{31}& z_{31}& -z_{37}& -z_{36}& z_{35}& z_{34}& -z_{33}& -z_{32}& z_{31}& -z_{31}\\
-z_{24}& z_{24}& z_{30}& z_{29}& -z_{28}& -z_{27}& z_{26}& z_{25}& z_{24}& -z_{24}\\
z_{17}& -z_{17}& z_{23}& z_{22}& -z_{21}& -z_{20}& z_{19}& z_{18}& -z_{17}& z_{17}\\
-z_{16}& z_{16}& z_{13}& z_{14}& -z_{15}& -z_{15}& -z_{14}& -z_{13}& z_{11}& z_{12}\\
z_{16}& -z_{16}& -z_{13}& -z_{14}& z_{15}& z_{15}& z_{14}& z_{13}& z_{12}& z_{11}
\end{pmatrix}.
$$

4. On all other blocks  $x_4$ has the form
$$
\begin{pmatrix}
z_{38}& z_{39}& z_{40}& z_{41}\\
z_{42}& z_{43}& z_{44}& z_{45}\\
-z_{45}& -z_{44}& z_{43}& z_{42}\\
-z_{41}& -z_{40}& z_{39}& z_{38}
\end{pmatrix}.
$$

Therefore, we have $85$ variables $y_1,\dots, y_{40}, z_1,\dots, z_{45}$, where $y_1$, $y_6$, $y_{11}$, $y_{20}$, $y_{21}$, $y_{30}$, $y_{32}$, $y_{36}$, $y_{39}$, $y_{44}$, $z_1$, $z_5$, $z_7$, $z_{11}$, $z_{18}$, $z_{26}$, $z_{28}$, $z_{30}$, $z_{38}$, $z_{34}$, $z_{43}$, $z_{45}$ are $1$ modulo radical,  $y_2$, $y_4$, $y_{17}$, $y_{46}$, $z_2$, $z_3$, $z_9$ are $-1$ modulo radical, $z_{32}$ is $-2$ modulo radical, all other variables are from radical.

We apply step by step four basis changes, commuting with each other and with all matrices $w_{i}$. These changes are represented by matrices $C_1$, $C_2$, $C_3$, $C_4$. Matrices $C_1$ and $C_2$ are block-diagonal, where first $24$ blocks have the size $2\times 2$, the last block is $4\times 4$. On all  $2\times 2$ blocks, corresponding to short roots, the matrix  $C_1$ is unit, on all  $2\times 2$ blocks, corresponding to long roots, it is
$$
\begin{pmatrix}
1& -y_{16}/y_{17}\\
-y_{16}/y_{17}& 1
\end{pmatrix}.
$$
On the last block it is unit.

Similarly,  $C_2$ is unit on the blocks corresponding to long roots, and on the last block. On the blocks corresponding to the short roots, it is
$$
\begin{pmatrix}
1& -z_{8}/z_{9}\\
-z_{8}/z_{9}& 1
\end{pmatrix}.
$$

 Matrices $C_3$ and $C_4$ are diagonal, identical on the last $4\times 4$ block, the matrix $C_3$ is identical on all places, corresponding to short root, and scalar with multiplier~$a$ on all places corresponding to long roots. In the contrary, the matrix  $C_4$, is identical on all places, corresponding to long roots, and is scalar with multiplier~$b$ on all places, corresponding to short roots.

Since all these four matrices commutes with all $w_{i}$, $i=1,2,3,4$, then after basis change with any of these matrices all conditions for elements $x_1$ and $x_4$  still hold.

At the beginning we apply basis changes with the matrices $C_1$ and $C_2$. After that new $y_{16}$ in the matrix $x_1$ and $z_8$ in the matrix $x_4$ are equal to zero (for the convenience of notations we do not change names of variables). Then we choose  $a=-1/y_{17}$ (it is new $y_{17}$) and apply the third basis change. After it $y_{17}$ in the matrix $x_1$ becomes to be $-1$. Clear that $y_{16}$ is still zero.

Finally, apply the last basis change with $b=-1/z_9$ (where $z_9$ is the last one, obtained after all previous changes). We have that $y_{16}, y_{17}, z_8$ are not changed, and $z_9$ is now $-1$.

Now we can suppose that $y_{16}=0, y_{17}=-1, z_8=0, z_9=-1$, we have now just $81$~variables.

From the fact that $x_1$ and $x_4$ commute (Cond.~1), it directly follows $y_{37}=y_{38}=0$, $y_{36}=y_{20}$.
From the condition $h_{\alpha_2}(-1)x_1 h_{\alpha_2}(-1) x1=E$ (Cond.~2, its position $(52,52)$) follows that $y_{20}^2=1$, consequently $y_{20}=1$.

From the condition $w_{2}x_1 w_{2}^{-1} x_1=x_1w_{2}(1)x_1 w_{2}(1)^{-1} $ (Cond.~3, the position $(50,10)$) it follows $y_{21}=1$, from its position $(49,10)$ it follows $y_{19}=0$.

The condition
$w_{2}w_{3} w_{2} x_1 w_{2}^{-1}w_{3}(1) w_{2}^{-1} x_1
=x_1 w_{2}w_{3} w_{2} x_1 w_{2}^{-1}w_{3}^{-1} w_{2}^{-1}$
 (Cond.~4, the position $(51,52)$) implies $y_{15}=0$.

Again from Cond.~3 (the position $(18,13)$) we have $y_{46}(y_{45}+y_{42})=0$, whence $y_{45}=-y_{42}$.
From Cond.~2 (the positions $(11,12)$ and $(12,11)$) we obtain $y_{40}(y_{39}+y_{44})=0$ and $y_{43}(y_{39}+y_{44})=0$, therefore $y_{40}=y_{43}=0$. After that in the same condition the position $(12,16)$ gives $y_{44}=y_{39}$. The position $(12,16)$ of Cond.~3 now gives us $y_{46}(y_{39}-1)=0 \Rightarrow y_{39}=1$.

 In the condition $h_{\alpha_3}(-1)x_4 h_{\alpha_3}(-1)x_4=E$ (Cond.~5) the position $(8,7)$ gives $z_4=0$, the position $(7,7)$ gives $z_1=1$;  $(51,51)$ gives $z_7=1$;

In the condition $w_{3}x_4 x_{3}^{-1} x_4=x_4w_{3}x_4 x_{3}^{-1}$ (Cond.~6) the position $(51,5)$ gives $z_{41}=0$, the position $(51,6)$ gives $z_{40}=0$, the position $(52,7)$ gives $z_{39}=0$, the position $(51,8)$ gives $z_{10}=0$, the position $(52,8)$ gives $z_{38}=1$.

Again from Cond.~5 (positions $(52,52)$, $(52,8)$, $(7,8)$) we obtain $z_6=0$, $z_5=1$, $z_2=z_3$.

Returning to Cond.~6, from  $(13,51)$ we have $z_{43}=1$, from $(5,51)$ we have $z_{44}=0$, from $(5,14)$ we have $z_{42}=0$, from $(12,17)$ we have $z_{35}=0$, from $(12,18)$ we have $z_{34}=1$, from $(12,19)$ --- $z_{37}=-z_{31}$, from $(12,20)$ --- $z_{36}=z_{31}$, from $(9,15)$ --- $z_{20}=-z_{15}$, and from  $(10,15)$ --- $z_{27}=z_{15}$.

The position $(11,22)$ of Cond.~1 now gives us $y_{42}=0$, and the position $(11,11)$ of Cond.~2 gives $y_{41}=0$.

 Considering $x_{1+2}=\varphi_2(x_{\alpha_1+\alpha_2}(1))=w_{2}x_1w_{2}^{-1}$, $x_2=\varphi(x_{\alpha_2}(1))=w_{1}x_{1+2}w_{1}$ and Cond.~7:  $x_1x_2=x_{1+2}x_2x_1$ (the position $(6,16)$), we obtain $y_{46}=-1$.

 Similarly, considering $x_{3+4}=\varphi_2(x_{\alpha_3+\alpha_4}(1))=w_{3}x_4w_{3}^{-1}$, $x_3=\varphi(x_{\alpha_3}(1))=w_{4}x_{3+4}w_{4}^{-1}$, and Cond.~8:  $x_3x_4=x_{3+4}x_4x_3$  (applying positions $(51,14)$, $(13,52)$, $(12,11)$, $(29,9)$, $(15,35)$, $(15,36)$, $(16,36)$, $(12,19)$, $(12,20)$, $(11,25)$, $(12,26)$, $(10,30)$, $(47,11)$, $(1,2)$, $(1,1)$, $(4,4)$, $(3,4)$, $(3,18)$, $(3,17)$, $(4,17)$, $(4,3)$, $(3,3)$, $(18,3)$), we obtain $z_{45}=1$, $z_3=-1$, $z_{31}=0$, $z_{32}=-2$, $z_{14}=0$, $z_{13}=0$, $z_{30}=1$, $z_{25}=0$, $z_{26}=1$, $z_{15}=0$, $z_{28}=1$, $z_{24}=0$, $z_{16}=0$, $z_{12}=0$, $z_{11}=1$, $z_{17}=0$, $z_{19}=0$, $z_{21}=0$, $z_{22}=0$, $z_{29}=0$, $z_{23}=0$, $z_{18}=1$, $z_{33}=0$, respectively.

Therefore we obtain that $x_4=x_{\alpha_4}(1)$.

Directly from the first condition we now have $y_3=y_7=y_{27}=y_{25}=y_{34}=y_{26}=y_{33}=y_{28}=y_{35}=y_{22}=y_{24}=y_{29}=y_{31}=y_{12}=y_{13}
=y_9=y_{10}=y_{23}=y_{18}=y_{14}=0$, $y_{30}=y_{32}=y_{11}=1$.

Finally, from Cond.~3 we get
$y_5=0$, $y_6=1$, $y_1=1$, $y_8=0$, $y_4=-1$, from Cond.~2 we get $y_2=-1$.

Now $x_1=x_{\alpha_1}(1)$, it is what we needed.

Since all long (and all short) roots are conjugate under the action of Weil group, it means that
$\varphi_2(x_\alpha(1))=x_\alpha(1)$ for all $\alpha\in \Phi$.

 Consider now the matrix
 $d_t=\varphi_2(h_{\alpha_4}(t))$.

\begin{lemma}\label{l4_1}
The matrix $d_t$ is $h_{\alpha_4}(s)$ for some $s\in R^*$.
\end{lemma}

\begin{proof}
Since the matrix $d_t$ commutes with $h_{\alpha}(-1)$ for all $\alpha\in \Phi$, then $d_t$ is decomposed to the following diagonal blocks:
\begin{align*}
D_1& =\{v_{1}, v_{-1}, v_{{14}}, v_{-{14}}, v_{{20}}, v_{-{20}},
v_{{22}}, v_{-{22}}\},\\
D_2& =\{v_{2}, v_{-2}, v_{{10}}, v_{-{10}}, v_{{16}}, v_{-{16}},
v_{{24}}, v_{-{24}}\},\\
D_3&= \{v_{3}, v_{-3}\},\quad D_4= \{v_{4}, v_{-4}\},\\
D_5& =\{v_{5}, v_{-5}, v_{{12}}, v_{-{12}}, v_{{18}}, v_{-{18}},
v_{{23}}, v_{-{23}}\},\\
D_6&= \{v_{6}, v_{-6}\},\quad D_7= \{v_{7}, v_{-7}\},\\
D_8&= \{v_{8}, v_{-8}\},\quad D_9= \{v_{9}, v_{-9}\},\\
D_{10}&= \{v_{{11}}, v_{-{11}}\},\quad D_{11}= \{v_{{13}}, v_{-{13}}\},\\
D_{12}&= \{v_{{15}}, v_{-{15}}\},\quad D_{13}= \{v_{{17}}, v_{-{17}}\},\\
D_{14}&= \{v_{19}, v_{-{19}}\},\quad D_{15}= \{v_{{21}}, v_{-{21}}\},\\
D_{16}&=\{ V_{1}, V_{2}, V_{3}, V_{4}\}.
\end{align*}

Using the fact that $d_t$ commutes with $w_{1}$, $w_{2}$, $w_{{13}}$ and $x_1$, we obtain that on the blocks $D_1,D_2,D_5$ the matrix $d_t$ has the form
$$
\begin{pmatrix}
t_1& 0& 0& 0& 0& 0& 0& 0\\
0& t_1& 0& 0& 0& 0& 0& 0\\
0& 0& t_8& 0& t_9& 0& 0& 0\\
0& 0& 0& t_{10}& 0& t_{11}& 0& 0\\
0& 0& t_{11}& 0& t_{10}& 0& 0& 0\\
0& 0& 0& t_9& 0& t_8& 0& 0\\
0& 0& 0& 0& 0& 0& t_1+2t_{13}& 0\\
0& 0& 0& 0& 0& 0& 0& t_1
\end{pmatrix};
$$
 on the blocks $D_3$, $D_6$, $D_8$, $D_{14}$ it is $\diag[t_2,t_3]$; on the blocks  $D_7$, $D_9$, $D_{10}$, $D_{15}$ it is $\diag[t_3,t_2]$, on the block $D_4$ it is
$$
\begin{pmatrix}
t_4& t_5\\
t_6& t_7
\end{pmatrix}
;
$$
on the blocks $D_{11}$, $D_{12}$, $D_{13}$ it has the form $\diag[t_{12},t_{12}]$; and on the last block it is
$$
\begin{pmatrix}
t_1& 0& 0& 0\\
0& t_1& 0& 0\\
0& 0& t_1& 0\\
0& 0& t_{13}& t_1-2t_{13}
\end{pmatrix}.
$$

Using the condition $w_{4}d_t w_{4}^{-1}d_t=E$, we obtain: from the position $(1,1)$ it follows $t_1^2=1$, consequently $t_1=1$, from $(52,52)$ it follows $(1-2t_{13})^2=1$, therefore $t_{13}=0$;  $(5,5)$ implies $t_3=1/t_2$;  $(7,8)$ implies $t_7(t_5+t_6)=0$, whence $t_6=-t_5$; from $(24,36)$ we have $t_8(t_9+t_{11})=0$, therefore $t_{11}=-t_9$; from  $(26,26)$ we have $t_{12}^2=1$, and then $t_{12}=1$.

Now consider the condition $w_{3}d_t w_{3}^{-1}=d_tw_{4}w_{3}d_t w_{3}^{-1}w_{4}^{-1}$. Its position $(13,14)$ gives $t_5=0$, the position $(5,5)$ gives $t_4=1/t_2^2$,  $(6,6)$ gives $t_7=t_2^2$;  $(3,19)$ gives $t_9=0$;  $(19,19)$ gives $t_{10}=1/t_8$.

 Finally, introduce $\varphi_2(h_{\alpha_3}(t))=w_{4}w_{3}d_t w_{3}^{-1}w_{3}^{-1}$, $\varphi_2(h_{\alpha_6}(t))=w_{2}\varphi_2(h_{\alpha_3}(t))w_{2}^{-1}$, $\varphi_2(h_{\alpha_{10}}(t))=\varphi_2(h_{\alpha_6}(t))\varphi_2(h_{\alpha_3}(t))$. Since $\varphi_2(h_{\alpha_{10}}(t))$ commutes with $x_{\alpha_8}(1)$, we obtain (the position $(9,6)$) that $t_8=t_2^2$.

Therefore, $\varphi_2(h_{\alpha_{4}}(t))=h_{\alpha_4}(1/t_2)$, and the lemma is proved.
\end{proof}

Clear, that this lemma holds also for images of all $h_{\alpha}(t)$, $\alpha\in \Phi$.

\section{Images of $x_{\alpha_i}(t)$,
proof of Theorem 1.}

We have shown that $\varphi_2(h_{\alpha}(t))=h_{\alpha}(s)$,
$\alpha\in \Phi$. Denote the mapping $t\mapsto s$ by $\rho: R^*
\to R^*$. Note that for $t\in R^*$
$\varphi_2(x_1(t))=\varphi_2(h_{\alpha_2}(t^{-1}) x_1(1)
h_{\alpha_2}(t))=h_{\alpha_2}(s^{-1}) x_1(1) h_{\alpha_2}(s)=
x_1(s)$. If $t\notin R^*$, then $t\in J$, i.\,e., $t=1+t_1$, where
$t_1\in R^*$. Then
$\varphi_2(x_1(t))=\varphi_2(x_1(1)x_1(t_1))=x_1(1)x_1(\rho(t_1))=
x_1(1+\rho(t_1))$. Therefore if we extend the mapping
$\rho$ to the whole~$R$ (by the formula $\rho(t):=1+\rho(t-1)$, $t\in
R$), we obtain $\varphi_2(x_1(t))=x_1(\rho(t))$ for all $t\in R$.
Clear that $\rho$ is injective, additive, and also  multiplicative on all invertible elements. Since every element of~$R$ is a sum of two invertible elements, we have that  $\rho$ is an isomorphism from the ring~$R$ onto some its subring~$R'$. Note that in this situation $C G(R) C^{-1}=G(R')$ for some matrix $C\in
\GL(V)$. Let us show that $R'=R$.

Denote matrix units by $E_{ij}$.

\begin{lemma}\label{porozhd}
The Chevalley group $G(R)$ generates the matrix ring $M_n(R)$.
\end{lemma}
\begin{proof}
The matrix $(x_{\alpha_1}(1)-1)^2$ has a unique nonzero element
$-2\cdot E_{12}$. Multiplying it to suitable diagonal matrices, we can obtain
an arbitrary matrix of the form $\lambda\cdot E_{12}$ (since $-2\in R^*$ and $R^*$ generates~$R$). Since the Weil group acts transitively on all roots of the same length, i.\.e., for every long root $\alpha_k$ there exists such  $w\in W$, that
$w(\alpha_1)=\alpha_k$, and then the matrix $\lambda E_{12}\cdot w$ has the form
$\lambda E_{1,2k}$,  and the matrix $w^{-1}\cdot \lambda E_{12}$ has the form
$\lambda E_{2k-1,2}$. Besides, with the help of the Weil group element, moving the first root to the opposite one, we can get the matrix unit
$E_{2,1}$. Taking now different combinations of the obtained elements, we can get an arbitrary element $\lambda E_{ij}$, $1\leqslant i,j\leqslant 48$, indices $i,j$ correspond to the numbers of long roots.

The matrix $(x_{\alpha_4}(1)-1)^2$ is $-2E_{7,8}+2E_{20,32}+2E_{24,36}+2E_{28,40}+2E_{31,19}+2E_{35,23}+2E_{39,27}$.
All matrix units in this sum, except the first one, are already obtained, therefore we can subtract them and get $E_{7,8}$.
Similarly to the longs roots, using the fact that all short roots are also conjugate under the action of the Weil groups, we obtain all $\lambda E_{ij}$, $1\leqslant i,j\leqslant 48$, indices $i,j$ correspond to the short roots.

Now subtract from the matrix $x_{\alpha_1}(1)-1$ suitable matrix units and obtain the matrix  $E_{49, 2}-2 E_{1,49}+ E_{1,50}$.
Multiplying it (from the right side) to $E_{2,i}$, $1\leqslant i\leqslant 48$, where $i$ corresponds to a long root, we obtain all
$E_{49, i}$, $1\leqslant i\leqslant 48$ for $i$ corresponding to the long roots. Multiplying these last elements from the left side to $w_{2}$, we obtain $E_{50, i}$, $1\leqslant i\leqslant 48$ for $i$, corresponding to the long roots; then by multiplying them from the left side to $w_{3}$ we obtain all $E_{51, i}$, $1\leqslant i\leqslant 48$ for $i$, corresponding to the long roots, and, similarly, $E_{52,i}$. Therefore, now we have all $E_{i,j}$, $49\leqslant i\leqslant 52$, $1\leqslant j\leqslant 52$, where $j$ correspond to the long roots.

Then $A=1/8(h_{\alpha_1}(-1)+E)\dots (h_{\alpha_4}(-1)+E)=E_{49,49}+E_{50,50}+E_{51,51}+E_{52,52}$, $B=A(w_1+\dots+w_4)A+2A=E_{49,50}+E_{50,49}+E_{50,51}+2E_{51,50}+E_{51,52}+E_{52,51}$, $C=B^2-A=E_{49,51}+2E_{50,50}+E_{50,52}+2E_{51,49}+2E_{51,51}+2E_{52,50}$, $C^2-B^2=2E_{52,50}$. So we have $E_{52,50}$ and then all $E_{i,j}$, $48<i,j\leqslant 52$, therefore all
 $E_{i,j}$, $1\leqslant i\leqslant 48$, $48< j\leqslant 52$, where $i$ corresponds to the long roots.

Then, taking the matrix $x_{\alpha_4}(t)$ and multiplying it from the left and right side to some suitable matrix units $E_{i,i}$,  we can obtain $E_{i,j}$, where  $i$ corresponds to the long root, $j$ corresponds to the short one. After that it becomes clear, how to get all matrix units $E_{i,j}$, $1\leqslant i,j \leqslant 48$ with the help of the Weil group. Finally, as above, we can obtain all  $E_{i,j}$, $1\leqslant i\leqslant 48$, $48< j\leqslant 52$, where $i$ correspond to the short roots, and so all matrix units.

\end{proof}

\begin{lemma}\label{Tema}
If for some  $C\in \GL_n(R)$ we have $C G(R) C^{-1}=
G(R')$, where $R'$ is a subring of~$R$, then $R'=R$.
\end{lemma}
\begin{proof}
Suppose that $R'$ is a proper subring of~$R$.

Then $C M_n(R) C^{-1} =M_n (R')$, since the group $G(R)$
generates the whole ring $M_n(R)$ (the previous lemma), and the group $G(R')=CG(R) C^{-1}$
generated the ring $M_n(R')$. It is impossible, since $C\in \GL_n(R)$.
\end{proof}

\emph{Proof of Theorem}~1.
We have just proved that  $\rho$ is an automorphism of the ring~$R$. Consequently, the composition of the initial automorphism~$\varphi$ and some basis change with a matrix $C\in \GL_n(R)$, (mapping  $G(R)$ into itself)
is a ring automorphism~$\rho$. It proves Theorem~1. $\square$

\section{Theorem about normalizers and Main Theorem}

To prove the main theorem of this paper (see Theorem~\ref{main} in the end of this section), we need to obtain the following important fact (that has proper interest):

\begin{theorem}\label{norm}
Every automorphism--conjugation of a Chevalley group~$G(R)$ of type $F_4$ over a local ring~$R$ with~$1/2$ is an inner automorphism.
\end{theorem}

\begin{proof}
Suppose that we have some matrix $C=(c_{i,j})\in \GL_{52}(R)$ such that
$$
C\cdot G \cdot C^{-1}=G.
$$

If $J$ is the radical of $R$, then  $M_n(J)$ is the radical in the matrix ring $M_n(R)$, therefore
$$
C\cdot M_n(J)\cdot C^{-1}=M_n(J),
$$
consequently,
$$
C\cdot (E+M_n(J))\cdot C^{-1}=E+M_n(J),
$$
i.\,e.,
$$
C\cdot G(R,J)\cdot C^{-1}=G(R,J),
$$
since $G(R,J)=G\cap (E+M_n(J)).$

Thus, the image $\overline C$ of the matrix~$C$ under factorization~$R$ by~$J$ gives us an automorphism--conjugation of the Chevalley group $G(k)$, where $k=R/J$ is a residue field of~$R$.

But over a field every automorphism--conjugation of a Chevalley group of type $F_4$ is inner (see~\cite{Steinberg}), therefore a conjugation by $\overline C$ (denote it by~$i_{\overline C}$) is
$$
i_{\overline C}= i_g,
$$
where $g\in G(k)$.

Since over a field our Chevalley group (of type $F_4$) coincides with its elementary subgroup, every its element is a product of some set of unipotents $x_\alpha(t)$) and the matrix  $g$ can be decomposed into a product $x_{\alpha_{i_1}}(Y_1)\dots x_{i_{N}}(Y_{N})$, где
$Y_1,\dots, Y_{N}\in k$.

Since every element $Y_1,\dots, Y_N$
is a residue class in~$R$, we can choose
(arbitrarily) elements $y_1\in Y_1$, \dots, $y_N\in Y_N$, and the element
$$
g'=x_{\alpha_{i_1}}(y_1)\dots x_{i_N}(y_N)
$$
satisfies  $g'\in G(R)$ and $\overline
{g'}=g$.

Consider the matrix $C'={g'}^{-1}\circ d^{-1}\circ C$. This matrix also normalizes the group $G(R)$, and
also $\overline {C'}=E$. Therefore, from the description of the normalizer of~$G(R)$ we come to the description of all matrices from this normalizer equivalent to the unit matrix modulo~$J$.

Therefore we can suppose that our initial matrix $C$ is equivalent to the unit modulo~$J$.

Our aim is to show that $C\in \lambda G(R)$.

Firstly we prove one technical lemma that we will need later.

\begin{lemma}\label{prod2}
Let $X=\lambda t_{\alpha_1}(s_1)\dots t_{\alpha_4}(s_4)x_{\alpha_1}(t_1)\dots x_{\alpha_{24}}(t_{24})x_{-\alpha_1}(u_1)\dots x_{-\alpha_{24}}(u_{24})\in \lambda G(R,J)$.
Then the matrix $X$ has such $53$ coefficients \emph{(}precisely described in the proof of lemma\emph{)}, that uniquely define all $s_1,\dots,s_4,t_1,\dots,t_{24},u_1,\dots, u_{24}, \lambda$.
\end{lemma}

\begin{proof}
Consider the sequence of roots:
\begin{align*}
&\gamma_1=\alpha_1,\\
&\gamma_2=\alpha_5=\alpha_1+\alpha_2,\\
&\gamma_3=\alpha_8=\alpha_1+\alpha_2+\alpha_3,\\
&\gamma_4=\alpha_{12}=\alpha_1+\alpha_2+2\alpha_3,\\
&\gamma_5=\alpha_{15}=\alpha_1+\alpha_2+2\alpha_3+\alpha_4,\\
&\gamma_6=\alpha_{17}=\alpha_1+2\alpha_2+2\alpha_3+\alpha_4,\\
&\gamma_7=\alpha_{19}=\alpha_1+2\alpha_2+3\alpha_3+\alpha_4,\\
&\gamma_8=\alpha_{21}=\alpha_1+2\alpha_2+3\alpha_3+2\alpha_4,\\
&\gamma_9=\alpha_{22}=\alpha_1+2\alpha_2+4\alpha_3+2\alpha_4,\\
&\gamma_{10}=\alpha_{23}=\alpha_1+3\alpha_2+4\alpha_3+2\alpha_4,\\
&\gamma_{11}=\alpha_{24}=2\alpha_1+3\alpha_2+4\alpha_3+2\alpha_4.
\end{align*}
All roots of~$F_4$, except $\alpha_{14}$ and $\alpha_{18}$, are differences between two distinct roots of this sequence (or its member).

Besides,  $\gamma_1$ is a simple root, $\gamma_{11}$ is a maximal root of the system, every root of the sequence is obtained from the previous one by adding some simple root.

Consider in the matrix $X$ some place $(\mu,\nu)$, $\mu,\nu\in \Phi$.

To find an element on this position we need to define all sequences of roots $\beta_1, \dots, \beta_p$, satisfying the following properties:

1. $\mu + \beta_1\in \Phi$, $\mu +\beta_1+\beta_2\in \Phi$, \dots, $\mu+\beta_1+\dots+\beta_i\in \Phi$, \dots,
$\mu+\beta_1+\dots+\beta_p=\nu$.

2. In the initial numerated sequence  $\alpha_1,\dots,\alpha_{24}, -\alpha_1,\dots, -\alpha_{24}$ the roots $\beta_1,\dots,\beta_k$ are replaced strictly from right to left.

Finally in the matrix $X$ on the position $(\mu,\nu)$ there is the sum of all products $\pm \beta_1\cdot\beta_2\dots \beta_p$ by all sequences with these two properties, multiplying to $d_\mu=\lambda s_1^{\langle \alpha_1,\mu\rangle}\dots s_4^{\langle \alpha_4,\mu\rangle}$. If $\mu=\nu$, we must add $1$ to the sum.

We will find the obtained elements $s_1,\dots,s_4,t_1,\dots, t_m,u_1,\dots,u_m$  step by step.

 Firstly we consider in the matrix $X$ the position $(-\gamma_{11},-\gamma_{11})$. We can not add to the root $-\gamma_{11}$ any negative root to obtain a root in the result. If in a sequence  $\beta_1,\dots, \beta_p$ the first root is positive, then all other roots must be positive. Thus, this position contains an element $1\cdot d_\nu$. So we know $d_{-\gamma_{11}}$. By the previous arguments if we consider the position $(-\gamma_{11},-\gamma_{10})$, the suitable sequence is only $\alpha_{1}=\gamma_{11}-\gamma_{10}$. Since there is $d_{-\gamma_{11}} t_1$ on this position and we already know $d_{-\gamma_{11}}$,  we can find $t_{1}$ on the position $(-\alpha_{24},-\alpha_{23})$.
Considering the positions $(-\gamma_{10},-\gamma_{10})$  and $(-\gamma_{10},-\gamma_{11})$, we see that by similar reasons there are $d_{-\gamma_{10}}(1\pm u_1t_1)$ and  $\pm d_{-\gamma_{10}} u_{1}$ there. So we find $d_{-\gamma_{10}}$ and $u_1$.

Now we come to the second step.
As we have written above, in the matrix $X$ on the position $(-\gamma_{10},-\gamma_9)$ there is $d_{-\gamma_{10}}(\pm t_{2}\pm u_{1}t_{5})$; on the position $(-\gamma_{9},-\gamma_{10})$ there is $d_{-\gamma_{9}}(\pm u_{2}\pm u_{5}t_{1})$; on the position $(-\gamma_{11},-\gamma_9)$ there is $\pm d_{-\gamma_{11}} t_{5}$ (the second summand is absent, since $\alpha_1$ is staying earlier than $\alpha_2$);  on the position $(-\gamma_9,-\gamma_{11})$ there is $d_{-\gamma_{9}}(\pm u_{5}\pm u_{2}u_{1})$; finally, on the position $(-\gamma_9,-\gamma_{9})$ there is $d_{-\gamma_{9}}(1+\pm u_{5}t_5\pm u_{2}t_{2})$.  From the position $(-\gamma_{11},-\gamma_9)$ we find $t_5$, then from the position $(-\gamma_{10},-\gamma_9)$ we find $t_2$, and from other three positions together we can know $u_2,u_5,d_{-\gamma_{9}}$. Therefore, now  we know $t_1, t_2, t_{5},u_1, u_2,u_{5}, d_{-\gamma_{9}}, d_{-\gamma_{10}}, d_{-\gamma_{11}}$.

On the third step we consider the positions $(-\gamma_{9},-\gamma_8)$ with $d_{-\gamma_{9}}(\pm t_3\pm u_2t_{6}\pm u_5t_{8})$, $(-\gamma_{8},-\gamma_9)$ with $d_{-\gamma_{8}}(\pm u_3\pm t_2u_{6}\pm t_5u_{8})$, $(-\gamma_{10},-\gamma_8)$ with $d_{-\gamma_{10}}(\pm t_{6}\pm u_1t_{8})$, $(-\gamma_{8},-\gamma_{10})$ with $d_{-\gamma_{8}}(\pm u_{6}\pm u_2u_3\pm t_1u_{8})$, $(-\gamma_{11},-\gamma_8)$ with $d_{-\gamma_{11}}(\pm t_{8}\pm  t_{5}t_{3})$, $(-\gamma_{8},-\gamma_{11})$ with $d_{-\gamma_{8}}(\pm u_{8}\pm u_3u_2u_1 \pm u_{6}u_1)$, and $(-\gamma_{8},-\gamma_8)$ with $d_{-\gamma_{8}}(1\pm u_3t_3\pm u_{5}t_{5}\pm u_{8}t_{8}\pm u_{8}t_{5}t_3)$. From these seven equations with seven unknown variables (all of them from radical) we can find all variables $t_3,u_3,t_6,u_6,t_8,u_8$ and $d_{-\gamma_{8}}$.

Similarly on the next step we consider the positions $(-\gamma_{8},-\gamma_7)$, $(-\gamma_{7},-\gamma_8)$, $(-\gamma_{9},-\gamma_7)$, $(-\gamma_{7},-\gamma_9)$, $(-\gamma_{10},-\gamma_7)$, $(-\gamma_{7},-\gamma_{10})$, $(-\gamma_{11},-\gamma_7)$, $(-\gamma_{7},-\gamma_{11})$, and $(-\gamma_{7},-\gamma_7)$, and find $t_4$, $u_4$, $t_7$, $u_7$, $t_9$, $u_9$, $t_{11}$, $u_{11}$, $d_{-\gamma_{7}}$.

Now we know $d_{-\gamma_{7}}, d_{-\gamma_{8}}, d_{-\gamma_{9}},d_{-\gamma_{10}}$ and $d_{-\gamma_{11}}$, i.\,e., $\lambda s_4/s_3$, $\lambda/s_4$, $\lambda s_2/s_3$, $\lambda s_1/s_2$ and $\lambda/s_1$. So we know all $s_i$, $i=1,\dots,4$, $\lambda$, and, consequently, all $d_{-\gamma_{i}}$.

Suppose now that we know all elements $t_i, u_j$ for all indices corresponding to the roots of the form $\gamma_p-\gamma_q$, $11\geqslant p,q> s$.
 Consider the positions $(-\gamma_{11},-\gamma_s)$, $(-\gamma_s,-\gamma_{11})$, $(-\gamma_{10},-\gamma_s)$, $(-\gamma_s,-\gamma_{10})$, \dots, $(-\gamma_{s+1},-\gamma_s)$,
$(-\gamma_s,-\gamma_{s+1})$ in the matrix~$X$. Clear that on every place $(-\gamma_i,-\gamma_s)$, $1\geqslant i>s$, there is sum of $t_p$, where $p$ is a number of the root $\gamma_i-\gamma_s$ (if it is a root), and products of different elements $t_a,u_b$, where only one member of the product is not known yet, all other elements are known and lie in radical; and all this sum is multiplying to the known element $d_{-\gamma_{i}}$. The same situation is on the positions $(-\gamma_s,-\gamma_i)$, $1\geqslant i>s$, but there is not $t_p$, but $u_p$ without multipliers here. Therefore, we have exactly the same number of (not uniform) linear equations as the number of roots of the form $\pm (\gamma_i-\gamma_s)$, with the same number of variables, in every equation exactly on variable has invertible coefficient, other coefficients are from radical, for distinct equations such variables are different. Clear that such a system has the solution, and it is unique. Consequently, we have made the induction step and now we know  elements $t_i, u_j$ for all indices, corresponding to the roots $\gamma_p-\gamma_q$, $11\geqslant p,q\geqslant s$.

On the last step we know elements $t_i, u_j$ for all indices, corresponding to the roots $\gamma_p-\gamma_q$, $11\geqslant p,q
\leqslant 1$. Consider now in $X$ the positions $(-\gamma_{11},h_{\gamma_{11}})$, $(h_{\gamma_{11}},-\gamma_{11})$, $(-\gamma_{10},h_{\gamma_{10}})$, $(h_{\gamma_{10}},-\gamma_{10})$, \dots, $(-\gamma_1,h_{\gamma_1})$, $(h_{\gamma_1},-\gamma_1)$. Similarly to the previous arguments we can find all  $t$ and $u$, corresponding to the roots $\pm \gamma_1,\dots, \pm \gamma_k$.

We have not found yet the obtained coefficients for two pairs of roots: $\pm \alpha_{14}$ and $\pm \alpha_{18}$.
Note that $\alpha_{14}+\alpha_{18}=\alpha_{24}$.

Consider in $X$ the positions $(-\alpha_{24},-\alpha_{14})$, $(-\alpha_{14}, -\alpha_{24})$, $(-\alpha_{24},-\alpha_{18})$, $(-\alpha_{18}, -\alpha_{24})$. On these positions there are sums of $t_{18}$ (respectively, $u_{18}$, $t_{14}$, $u_{14}$),  and products of elements $t_i,u_j$, corresponding to roots of smaller heights.  Since for all heights smaller than the height of $\alpha_{14}$, we know  $t,u$, then we can directly find the obtained coefficients.

Therefore, lemma is completely proved.
\end{proof}

Now return to our main proof.
Recall that we work with a matrix~$C$, equivalent to the unit matrix modulo radical, and normalizing Chevalley group $G(R)$.

For every root $\alpha\in\Phi$ we have
\begin{equation}\label{osn_eq}
C x_{\alpha}(1)C^{-1}=x_{\alpha}(1)\cdot g_\alpha,\quad g_\alpha\in
G(R,J).
\end{equation}
Every $g_\alpha \in G(R,J)$ can be decomposed into a product
\begin{equation}\label{razl_rad}
 t_{\alpha_1}(1+a_1)\dots
t_{\alpha_4}(1+a_4)x_{\alpha_1}(b_1)\dots
x_{\alpha_{24}}(b_{24})x_{\alpha_{-1}}(c_1)\dots x_{\alpha_{-24}}(c_{24}),
\end{equation}
where $a_1,\dots,a_4,b_1,\dots,b_{24},c_1,\dots, c_{24}\in J$ (see,
for example,~\cite{Abe1}).

Let $C=E+X=E+(x_{i,j})$. Then for every root~$\alpha\in \Phi$
we can write a matrix equation~\eqref{osn_eq} with variables $x_{i,j},
a_1,\dots,a_4,b_1,\dots,b_{24},c_1,\dots, c_{24}$, every of them is from radical.

Let us change these equations.
We consider the matrix~$C$ and ``imagine'', that it is some matrix from Lemma~\ref{prod2} (i.\,e., it is from $\lambda G(R)$). Then by some its concrete $53$~positions we can  ``define'' all coefficients $\lambda, s_1,\dots, s_4, t_1,\dots,t_{24},u_1,\dots, u_{24}$ in the decomposition of this matrix from Lemma~\ref{prod2}. In the result we obtain a matrix $D\in \lambda G(R)$, every matrix coefficient in it is some (known) function of coefficients of~$C$.
Change now the equations~\eqref{osn_eq} to the equations
\begin{equation}\label{fol_eq}
D^{-1}C x_{\alpha}(1)C^{-1}D=x_{\alpha}(1)\cdot {g_\alpha}',\quad {g_\alpha}'\in
G(R,J).
\end{equation}
We again have matrix equations, but with variables $y_{i,j},
a_1',\dots,a_4',b_1',\dots,b_{24}',c_1',\dots, c_{24}'$, every of them still is from radical, and also every
 $y_{p,q}$ is some known function of (all) $x_{i,j}$. The matrix $D^{-1}C$ will be denoted by~$C'$.

We want to show that a solution exists only for all variables with  primes equal to zero. Some $x_{i,j}$ also will equal to zero, and other are reduced in the equations. Since the equations are very complicated we will consider the linearized system. It is sufficient to show that all variables from the linearized system (let it be the system of  $q$~variables) are members of some system from $q$ linear equations with invertible in~$R$ determinant.

In other words, from the matrix equalities we will show that all variables from them are equal to zeros.

Clear that  linearizing  the product $Y^{-1}(E+X)$ we obtain some matrix $E+(z_{i,j})$, with all positions described in Lemma~\ref{prod2} equal to zero.

To find the final form of the linearized system, we write it as follows:
\begin{multline*}
(E+Z)x_\alpha(1) =x_\alpha(1)(E+a_1T_1+a_1^2\dots)\dots
(E+a_4T_l+a_4^2\dots)\cdot\\
\cdot(E+b_1X_{\alpha_1}+b_1^2X_{\alpha_1}^2/2)\dots
(E+c_{24}X_{-\alpha_{24}}+c_{24}^2X_{-\alpha_{24}}^2/2)(E+Z),
\end{multline*}
where $X_\alpha$ is a corresponding Lie algebra element in the adjoint representation, the matrix $T_i$ is diagonal, has on its diagonal $\langle \alpha_i,\alpha_k$ on the place corresponding to $v_k$; on the places corresponding to the vectors $V_j$, this matrix has zeros.

Then the linearized system has the form
$$
Zx_{\alpha}(1)-x_{\alpha}(1)(Z+a_1T_1+\dots+a_4T_4+b_1X_{\alpha_1}+\dots+c_{24}X_{\alpha_{24}})=0.
$$
This equation can be written for every $\alpha\in
\Phi$ (naturally, with another $a_j, b_j, c_j$), and can be written only for generating roots: for $\alpha_1,\dots,
\alpha_4, -\alpha_1, \dots, -\alpha_4$:

$$
\begin{cases}
Zx_{\alpha_1}(1)-x_{\alpha_1}(1)(Z+a_{1,1}T_1+\dots +a_{4,1}T_4+\\
\ \ \ \ \ b_{1,1}X_{\alpha_1}+
b_{2,1}X_{\alpha_2}+\dots+b_{24,1}X_{\alpha_{24}}
+c_{1,1}X_{-\alpha_1}+\dots+c_{24,1}X_{-\alpha_{24}})=0;\\
\dots\\
Zx_{\alpha_4}(1)-x_{\alpha_4}(1)(Z+a_{1,4}T_1+\dots +a_{4,4}T_4+\\
\ \ \ \ \
b_{1,4}X_{\alpha_1}+\dots+X_{\alpha_{24}}b_{24,1}X_{\alpha_{24}}
+c_{1,4}X_{-\alpha_1}+\dots+c_{24,4}X_{-\alpha_{24}})=0;\\
\dots\\
Zx_{-\alpha_1}(1)-x_{-\alpha_1}(1)(Z+a_{1,5}T_1+\dots+a_{4,5}T_4+\\
\ \ \ \ \
+b_{1,5}X_{\alpha_1}+\dots+b_{24,5}X_{\alpha_{24}}
+c_{1,5}X_{-\alpha_1}+\dots+c_{24,5}X_{-\alpha_5})=0;\\
\dots\\
Zx_{-\alpha_4}(1)-x_{-\alpha_4}(1)(Z+a_{1,8}T_1+\dots+a_{4,8}T_4+\\
\ \ \ \ \
+b_{1,8}X_{\alpha_1}+\dots+b_{24,8}X_{\alpha_{24}}
+c_{1,8}X_{-\alpha_1}+\dots+c_{24,8}X_{-\alpha_{24}})=0.
\end{cases}
$$

The matrix $T_1$ is
\begin{multline*}
\diag[2,-2,-1,1,0,0,0,0,1,-1,-1,1,0,0,1,-1,-1,1,-1,1,1,-1,\\
1,-1,-1,1,0,0,1,-1,-1,1,0,0,1,-1, 0,0,0,0,0,0,0,0,-1,1,1,-1,0,0,0,0];
\end{multline*}
$T_2$ is $w_{1}w_{2}T_1w_{2}^{-1}w_{1}^{-1}$;
 $T_3$ is
\begin{multline*}
\diag[0,0,-2,2,2,-2,-1,1,-2,2, 0,0,1,-1,0,0,-1,1,2,-2, -1,1,\\
2,-2,1,-1,0,0,1,-1, 0,0,-1,1,0,0, 1,-1,-2,2,  0,0,2,-2,0,0,0,0,0,0,0,0];
\end{multline*}
the matrix $T_4$ is $w_{3}w_{4}T_3w_{4}^{-1}w_{3}^{-1}$.

 The matrices $X_{\alpha_1}$, $X_{\alpha_3}$ were written above. Besides them, $X_{-\alpha_1}=w_{1}X_{\alpha_1}w_{1}^{-1}$, $X_{-\alpha_3}=w_{3}X_{\alpha_3}w_{3}^{-1}$. Other matrices  $X_\alpha$ are obtained as follows: $X_{\pm \alpha_5}=w_{2}X_{\pm \alpha_1} w_{2}^{-1}$, $X_{\pm \alpha_2}=w_{1}X_{\pm \alpha_5}w_{1}^{-1}$, $X_{\pm \alpha_{10}}=w_{3}X_{\pm \alpha_2}w_{3}^{-1}$, $X_{\pm \alpha_{12}}=w_{1}X_{\pm \alpha_{10}}w_{1}^{-1}$, $X_{\pm \alpha_{14}}=w_{2}X_{\pm \alpha_{12}}w_{2}^{-1}$, $X_{\pm \alpha_{16}}=w_{4}X_{\pm \alpha_{10}}w_{4}^{-1}$, $X_{\pm \alpha_{18}}=w_{1}X_{\pm \alpha_{16}}w_{1}^{-1}$, $X_{\pm \alpha_{20}}=w_{2}X_{\pm \alpha_{18}}w_{2}^{-1}$, $X_{\pm \alpha_{22}}=w_{3}X_{\pm \alpha_{20}}w_{3}^{-1}$, $X_{\pm \alpha_{23}}=w_{2}X_{\pm \alpha_{22}}w_{2}^{-1}$, $X_{\pm \alpha_{24}}=w_{1}X_{\pm \alpha_{23}}w_{1}^{-1}$, $X_{\pm \alpha_7}=w_{4}X_{\pm \alpha_3}w_{4}^{-1}$, $X_{\pm \alpha_4}=w_{3}X_{\pm \alpha_7}w_{3}^{-1}$, $X_{\pm \alpha_6}=w_{2}X_{\pm \alpha_3}w_{2}^{-1}$, $X_{\pm \alpha_8}=w_{1}X_{\pm \alpha_6}w_{1}^{-1}$, $X_{\pm \alpha_9}=w_{4}X_{\pm \alpha_6}w_{4}^{-1}$, $X_{\pm \alpha_{11}}=w_{1}X_{\pm \alpha_9}w_{1}^{-1}$, $X_{\pm \alpha_{13}}=w_{3}X_{\pm \alpha_9}w_{3}^{-1}$, $X_{\pm \alpha_{15}}=w_{1}X_{\pm \alpha_{13}}w_{1}^{-1}$, $X_{\pm \alpha_{17}}=w_{2}X_{\pm \alpha_{15}}w_{2}^{-1}$, $X_{\pm \alpha_{19}}=w_{3}X_{\pm \alpha_{17}}w_{3}^{-1}$,
$X_{\pm \alpha_{21}}=w_{4}X_{\pm \alpha_{19}}w_{4}^{-1}$.

From Lemma \ref{prod2} we obtain  that the following positions of $Z$ are zeros:
$(48,48)$, $(48,46)$, $(46,46)$, $(46,48)$, $(46,44)$, $(44,44)$, $(44,46)$, $(44,42)$, $(42,42)$, $(42,44)$, $(42,38)$, $(38,38)$, $(38,42)$, $(48,44)$, $(44,48)$, $(46,42)$, $(42,46)$, $(44,38)$, $(38,44)$, $(48,42)$, $(42,48)$, $(46,38)$, $(38,46)$, $(24,2)$, $(2,24)$, $(48,38)$, $(38,48)$, $(24,49)$, $(49,24)$, $(46,34)$, $(34,46)$, $(48,36)$, $(36,48)$, $(48,34)$, $(34,48)$, $(44,24)$, $(24,44)$, $(48,30)$, $(30,48)$, $(48,28)$, $(28,48)$, $(38,51)$, $(51,38)$, $(48,24)$, $(24,48)$, $(48,16)$, $(16,48)$, $(48,10)$, $(10,48)$, $(48,2)$, $(2,48)$, $(48,49)$, $(49,48)$.

Suppose that we fixed the obtained uniform linear system of equation. Recall that our aim is to show that all values $z_{i,j}$, $a_{s,t}, b_{s,t}, c_{s,t}$ are equal to zero.

Consider the first condition.
It implies $a_{4,1}=0$ (pos.~$(42,42)$);  $a_{1,1}=0$ (pos.~$(48,48)$); $a_{3,1}=0$ (pos.~$(38,38)$); $a_{2,1}=0$ (pos.~$(39,39)$). Therefore,  $T_1,T_2,T_3,T_4$ do not entry to this condition. Later,     $c_{1,1}=0$ (pos.~$(3,9)$);    $b_{2,1}=0$ (pos.~$(3,51)$); $c_{2,1}=0$ (pos.~$(46,44)$);  $b_{3,1}=0$ (pos.~$(5,51)$); $c_{3,1}=0$ (pos.~$(6,51)$); $b_{4,1}=0$ (pos.~$(7,51)$); $c_{4,1}=0$ (pos.~$(8,51)$); $b_{5,1}=0$ (pos.~$(44,48)$); $c_{5,1}=0$ (pos.~$(10,51)$); $b_{6,1}=0$ (pos.~$(3,6)$); $c_{6,1}=0$ (pos.~$(46,42)$); $b_{7,1}=0$ (pos.~$(13,51)$); $c_{7,1}=0$ (pos.~$(14,51)$); $b_{8,1}=0$ (pos.~$(42,48)$);  $c_{8,1}=0$ (pos.~$(16,52)$); $b_{9,1}=0$ (pos.~$(17,51)$); $c_{9,1}=0$ (pos.~$(46,38)$); $b_{10,1}=0$ (pos.~$(19,51)$); $b_{11,1}=0$ (pos.~$(38,48)$); $c_{11,1}=0$ (pos.~$(22,51)$);  $c_{12,1}=0$ (pos.~$(24,51)$);
$b_{13,1}=0$ (pos.~$(25,51)$); $c_{13,1}=0$ (pos.~$(46,34)$);  $b_{14,1}=0$ (pos.~$(27,52)$); $c_{14,1}=0$ (pos.~$(28,51)$); $b_{15,1}=0$ (pos.~$(34,48)$); $c_{15,1}=0$ (pos.~$(30,51)$); $b_{16,1}=0$ (pos.~$(31,52)$); $c_{16,1}=0$ (pos.~$(46,28)$); $b_{17,1}=0$ (pos.~$(33,51)$); $c_{17,1}=0$ (pos.~$(34,51)$); $b_{18,1}=0$ (pos.~$(20,44)$); $c_{18,1}=0$ (pos.~$(36,52)$); $b_{19,1}=0$ (pos.~$(37,51)$);  $c_{19,1}=0$ (pos.~$(38,51)$); $b_{20,1}=0$ (pos.~$(39,51)$); $c_{20,1}=0$ (pos.~$(40,51)$); $b_{21,1}=0$ (pos.~$(41,52)$); $c_{21,1}=0$ (pos.~$(42,52)$);   $b_{22,1}=0$ (pos.~$(43,51)$); $c_{22,1}=0$ (pos.~$(44,51)$); $b_{23,1}=0$ (pos.~$(3,44)$); $c_{24,1}=0$ (pos.~$(10,43)$).

Consequently the right side of the condition contains only $X_{\alpha_{12}}$, $X_{\alpha_{24}}$, $X_{-\alpha_{10}}$, $X_{-\alpha_{23}}$, the condition itself is simplified, many elements of~$Z$ are equal to zero. Firstly, these are elements on the positions $(i,j)$, $i=2,3,5,6,7,8,10,11,13,14$, $16,17,19,22,24$, $25,27,28,30,31$, $33,36,37,38$, $39,40,41,42$, $43,44,45,48,50,51,52$, $j=1,4,9,12,15$, $18,20,21,23$, $26,29,32$, $35,46,47,49$ (except $z_{6,15}=c_{10,1}$, $z_{5,12}=b_{12,1}$, $z_{7,29}=c_{10,1}$, $z_{8,26}=b_{12,1}$, $z_{24,49}=-c_{10,1}$, $z_{28,35}=c_{23,1}$, $z_{27,32}=b_{24,1}$,  $z_{33,26}=-b_{24,1}$, $z_{34,29}=-c_{23,1}$, $z_{37,18}=b_{24,1}$, $z_{38,21}=c_{23,1}$, $z_{38,18}=c_{10,1}$, $z_{39,47}=-c_{10,1}$, $z_{39,20}=b_{24,1}$, $z_{40,23}=c_{23,1}$, $z_{41,12}=-b_{24,1}$, $z_{42,15}=-c_{23,1}$, $z_{43,4}=b_{24,1}$, $z_{44,9}=c_{23,1}$, $z_{45,49}=-b_{24,1}$).

When we make these elements equal to zero, we see that $b_{12,1}=0$ (pos.~$(19,2)$), $c_{10,1}=0$ (pos.~$(44,36)$, $b_{24,1}=0$ (pos.~$(45,2)$), $c_{23,1}=0$ (pos.~$(48,2)$, i.e., the condition now looks as $x_{\alpha_1}(1)Z=Zx_{\alpha_1}(1)$.
 By similar way finally all our conditions become of the form $x_{\pm\alpha_p}(1)Z=Zx_{\pm\alpha_p}(1)$, $p=1,\dots,4$. Since the centralizer of the given eight matrices consists of scalar matrices, and the matrix  $Z$ has a zero element $z_{52,52}$, we have that $Z=0$, what we need.

 Theorem~\ref{norm} is proved.
\end{proof}

From Theorems 1 and 2 directly follows the main theorem of the paper:

\begin{theorem}\label{main}
Let $G(R)$ be a Chevalley group with root system $F_4$, where $R$ is a local ring with~$1/2$. Then every automorphism of~$G(R)$ is standard, i.\,e., it is a composition of ring and inner automorphisms.
\end{theorem}

\end{document}